%% file: main.tex
\documentclass[12pt,reqno]{amsart}
\usepackage[utf8]{inputenc}

\usepackage{hyperref,color,polynom}
\hypersetup{
    colorlinks=true,
    linkcolor=blue,
    filecolor=magenta,      
    urlcolor=cyan,
}

\usepackage[title]{appendix}
\usepackage{enumitem}
\usepackage{comment}
\usepackage{float}
\usepackage[]{algorithm2e}

\usepackage{amsmath,amssymb,amsthm}
\usepackage{mathtools}
\usepackage{graphicx}
\usepackage[margin=1in]{geometry}
\usepackage{fancyhdr}
\usepackage{tabularx}
\usepackage{array}
\newcolumntype{L}{>{\arraybackslash}m{4.5cm}}

\theoremstyle{plain}
\newtheorem{theorem}{Theorem}[section]
\newtheorem{proposition}{Proposition}[section]
\newtheorem{lemma}{Lemma}[section]
\newtheorem{corollary}{Corollary}[section]

\theoremstyle{definition}
\newtheorem{definition}{Definition}[section]

\newtheorem{example}{Example}[section]
\newtheorem{alg}{Algorithm}[section]
\newtheorem{question}{Question}
\newtheorem{remark}{Remark}[section]

\newcommand{\comm}[1]{}

\newcommand{\lboxed}[1]{\begin{array}{|l}\hline#1\\\hline\end{array}}

\makeatletter
\newcommand{\llboxed}[1]{{%
  \let\@frameb@x\l@frameb@x\fbox{#1}%
}}
\newcommand{\rrboxed}[1]{{%
  \let\@frameb@x\r@frameb@x\fbox{#1}%
}}

\def\l@frameb@x#1{%
  \@tempdima\fboxrule
  \advance\@tempdima\fboxsep
  \advance\@tempdima\dp\@tempboxa
  \hbox{%
    \lower\@tempdima\hbox{%
      \vbox{%
        \hrule\@height\fboxrule
        \hbox{%
          \vrule\@width\fboxrule
          #1%
          \vbox{%
            \vskip\fboxsep
            \box\@tempboxa
            \vskip\fboxsep}%
          #1%
                    }%
        \hrule\@height\fboxrule}%
                          }%
        }%
}
\def\r@frameb@x#1{%
  \@tempdima\fboxrule
  \advance\@tempdima\fboxsep
  \advance\@tempdima\dp\@tempboxa
  \hbox{%
    \lower\@tempdima\hbox{%
      \vbox{%
        \hrule\@height\fboxrule
        \hbox{%
          #1%
          \vbox{%
            \vskip\fboxsep
            \box\@tempboxa
            \vskip\fboxsep}%
          #1%
          \vrule\@width\fboxrule}%
        \hrule\@height\fboxrule}%
                          }%
        }%
}
\makeatother

\renewcommand{\l}{\ell}

\title{Summand minimality and asymptotic convergence of generalized Zeckendorf decompositions}

\author{Katherine Cordwell}
\email{\textcolor{blue}{\href{mailto:ktcordwell@gmail.com}{ktcordwell@gmail.com}}}
\address{Department of Mathematics, University of Maryland, College Park, MD 20742}

\author{Max Hlavacek}
\email{\textcolor{blue}{\href{mailto:mhlavacek@g.hmc.edu}{mhlavacek@g.hmc.edu}}}
\address{Department of Mathematics, Harvey Mudd College, Claremont, CA 91711}

\author{Chi Huynh}
\email{\textcolor{blue}{\href{mailto:nhuynh30@gatech.edu,huynhngocyenchi@gmail.com}{nhuynh30@gatech.edu,huynhngocyenchi@gmail.com}}}
\address{School of Mathematics, Georgia Institute of Technology, Atlanta, GA 30332}

\author{Steven J. Miller}
\email{\textcolor{blue}{\href{mailto:sjm1@williams.edu, Steven.Miller.MC.96@aya.yale.edu}{sjm1@williams.edu,Steven.Miller.MC.96@aya.yale.edu}}}
\address{Department of Mathematics and Statistics, Williams College, Williamstown, MA 01267}

\author{Carsten Peterson}
\email{\textcolor{blue}{\href{mailto:carstenp@umich.edu}{carstenp@umich.edu}}}
\address{Department of Mathematics, University of Michigan, Ann Arbor, MI 48109}

\author{Yen Nhi Truong Vu}
\email{\textcolor{blue}{\href{mailto:ytruongvu17@amherts.edu}{ytruongvu17@amherts.edu}}}
\address{Department of Mathematics, Amherst College, Amherst, MA 01002}

\begin{document} 

\thanks{The fourth named author was partially supported by NSF Grants DMS1561945 and DMS1265673, while the last named author was supported by Professor Amanda Folsom and her NSF Grant DMS1449679. The authors thank the SMALL REU Program at Williams College, which is supported by NSF Grant DMS1347804 and the Williams College Science Center. The authors are also thankful for the detailed and helpful comments of the referee, which have greatly improved the exposition of this paper and allowed us to better connect our results with the existing literature. We would also like to thank Jeffrey Lagarias for recommending several relevant papers to us.}

\maketitle
\input{abstract.tex}
\input{introduction.tex}
\input{preliminaries.tex}
\input{weakly_decreasing.tex}

\input{preliminaries_2.tex}

\input{convergence.tex}
\input{summand_minimality.tex}
\input{conclusion.tex}

\input{appendix.tex}

\bibliographystyle{alpha}
\bibliography{biblio}

\end{document}

%% file: abstract.tex
\begin{abstract}
Given a recurrence sequence $H$, with $H_n = c_1 H_{n-1} + \dots + c_t H_{n-t}$ where $c_i \in \mathbb{N}_0$ for all $i$ and $c_1, c_t \geq 1$, the generalized Zeckendorf decomposition (gzd) of $m \in \mathbb{N}_0$ is the unique representation of $m$ using $H$ composed of blocks lexicographically less than $\sigma = (c_1, \dots, c_t)$. We prove that the gzd of $m$ uses the fewest number of summands among all representations of $m$ using $H$, for all $m$, if and only if $\sigma$ is weakly decreasing. We develop an algorithm for moving from any representation of $m$ to the gzd, the analysis of which proves that $\sigma$ weakly decreasing implies summand minimality. We prove that the gzds of numbers of the form $v_0 H_n + \dots + v_\ell H_{n-\ell}$ converge in a suitable sense as $n \to \infty$; furthermore we classify three distinct behaviors for this convergence. We use this result, together with the irreducibility of certain families of polynomials, to exhibit a representation with fewer summands than the gzd if $\sigma$ is not weakly decreasing.
\end{abstract}

%% file: introduction.tex
\section{Introduction}

Base-$d$ number systems, with $d \in \mathbb{N}_{\geq 2}$ have been used for millennia. A natural question is whether there exist number systems whose base is non-integral, or whose terms are not simply a geometric progression. Two closely related approaches to answering this question are $\beta$-expansions and linear numeration systems.

If $\beta \in \mathbb{R}_{> 1}$, then given $x \in \mathbb{R}_+$, one may form the \textit{$\beta$-expansion} of $x$ by first finding the largest $k \in \mathbb{Z}$ such that $\beta^k \leq x < \beta^{k+1}$, then recording $x_k = \lfloor x/\beta^k \rfloor$, and finally repeating the above procedure on $x - x_k \beta^k$. One obtains a representation $x = x_k \beta^k + x_{k-1} \beta^{k-1} + \cdots$. R\'{e}nyi \cite{Re} first introduced $\beta$-expansions in the more general context of representing numbers using monotonic functions (here $f(x) = x/\beta$), and Parry \cite{Pa} subsequently proved many fundamental results specific to $\beta$-expansions. They have been studied extensively in ergodic theory and dynamical systems (see, e.g., \cite{Bl}) with many interesting connections to number theory (see, e.g., \cite{Be, Schm}). 

An alternative line of thought originates from recognizing $\{1, d, d^2, \dots \}$ as an order-one linear recurrence sequence $H = \{H_n\}_{n = 0}^\infty$, with $H_0 = 1$ and $H_n = d H_{n-1}$. In general, given a linear recurrence sequence $H$, one may study representations of integers as sums of non-negative integral multiples of terms in $H$. In the computer science literature, these are known as linear numeration systems, and they have been studied largely from an automata theory perspective (see, e.g., \cite{Fr, Sh}). However, such number systems have appeared under a variety of names, such as $G$-ary expansions (with $G$ a linear recurrence sequence, see \cite{PT,GT, GTNP}), or without a clear name (e.g., \cite{Fra}). One of the earliest and most celebrated results on linear numeration systems is Zeckendorf's theorem, which states that every integer can be uniquely expressed as a sum of nonconsecutive Fibonacci numbers (provided we define them by $F_1 = 1$, $F_2 = 2$, and $F_{n+1} = F_n + F_{n-1}$). Though the theorem is often attributed to Zeckendorf \cite{Ze}, the result was proven earlier by Lekkerkerker \cite{Le} and is implicit in the work of Ostrowski \cite{Os}. 

It is immediate to see that the usual base-$d$ representation is the result of greedily representing a number using powers of $d$. Thus, given a recurrence sequence, it is natural to study the representation obtained via the greedy algorithm, which is known as the normal representation (see, e.g., \cite{Fr}). Alternatively, one may view the base-$d$  number system as only permitting representations which are built out of the digits $\{0, 1, \dots, d-1\}$. Thus, given a recurrence sequence, one may instead seek simple rules designating which representations are considered allowable, and which guarantee existence and uniqueness of a representation for every integer (for example, only considering representations built from a fixed family of digits). A particularly natural such rule is called the \textit{generalized Zeckendorf decomposition (gzd)} and is the focus of this paper. While in some cases the gzd is the same as the normal representation, they are in general not the same (see Section \ref{section_greedy}).

The gzd is only defined for linear recurrence sequences of the form $H_n = c_1 H_{n-1} + \dots + c_{t} H_{n-t}$ with $c_1, \ c_t \geq 1$ and $c_i \in \mathbb{N}_0$, and generated by ``ideal'' initial conditions $H_{-(t-1)} = \dots = H_{-1} = 0$ and $H_0 = 1$. We call such a sequence a \textit{positive linear recurrence sequence (PLRS)}. We call $\sigma = (c_1, \dots, c_t)$ the \textit{signature} of the recurrence sequence; given $\sigma$, we let $H_\sigma$ denote the corresponding recurrence sequence. The gzd was introduced for PLRSs independently by Hamlin and Webb \cite{HW}, and Miller and Wang \cite{MW}, though smaller classes of signatures had been considered previously \cite{Fra, GTNP}.  Several authors have given evidence that this is the largest class of signatures for which Zeckendorf's theorem extends in a simple way \cite{HW, CFHMN}. One main reason is because if $c_1 = 0$ or $c_i < 0$, then $H_\sigma$ is no longer strictly increasing and $x^t - c_1 x^{t-1} - \dots - c_t$ no longer necessarily has a unique dominating positive root (see also Proposition \ref{prop_pos_type}).

Classically many authors have studied the number of summands in the Zeckendorf decomposition (e.g., \cite{Le}, among others). More recently, various authors have studied questions related to the number of summands in the generalized Zeckendorf decomposition (\cite{PT, GT, GTNP, MW}, among others).

We are particularly interested in the number of summands in the gzd as compared to other representations of the same integer. We call a representation of $m$ \textit{summand minimal} if no other representation of $m$ uses fewer summands. We say that a PLRS $H$ is \textit{summand minimal} if the gzd using $H$ is summand minimal for all $m$. We completely classify which PLRSs are summand minimal.

\begin{theorem}\label{thm_summand_minimal}
A positive linear recurrence sequence with signature $(c_1, \dots, c_t)$ is summand minimal if and only if $c_1 \geq c_2 \geq \dots \geq c_t$. 
\end{theorem}

Let $\textnormal{Fin}(\beta)$ denote the set of real numbers with finite $\beta$-expansions. Theorem \ref{thm_summand_minimal} is strikingly similar to the following result of Frougny and Solomyak \cite{FS}.

\begin{theorem}[Theorem 2 of \cite{FS}] \label{thm_fs}
Let $\beta$ be the dominating root of a polynomial of the form $x^t - c_1 x^{t-1} - \dots - c_t$ with $c_i \in \mathbb{N}$ and $c_1 \geq c_2 \geq \dots \geq c_t \geq 1$. Then,
\begin{equation}
\textnormal{Fin}(\beta)\ =\ \mathbb{Z}[\beta^{-1}]_+\ :=\ \mathbb{Z}[\beta^{-1}] \cap \mathbb{R}_+. \label{eqn_f_prop}
\end{equation}
\end{theorem}

\noindent Summand minimality is closely related to (and may be seen as a strengthening of) Equation \eqref{eqn_f_prop}. In Section \ref{section_conclusion} we more thoroughly explain the relationship between Theorem \ref{thm_summand_minimal} and Theorem \ref{thm_fs}.

In Section \ref{section_prelim}, we discuss the generalized Zeckendorf theorem for PLRSs, and introduce the notions of borrowing and carrying. In Section \ref{section_weakly_decreasing}, we show that $\sigma$ weakly decreasing implies summand minimality by exhibiting a simple algorithm for moving from an arbitrary representation to the gzd when $\sigma$ is weakly decreasing. Analysis of this algorithm shows that the number of summands never increases along the way to the gzd, implying summand minimality of $H_\sigma$. For the sake of moving the reader forward towards the techniques and results of the second half, which we believe to be of greater significance, we relegate the description of a more general algorithm (Algorithm \ref{alg_gzd}) for moving between representations which works for all signatures, as well as a proof of its validity and termination, to Appendix \ref{section_algorithm}.

To prove that summand minimality implies a weakly decreasing signature, we turn our attention in a different direction. Given a finite string $v = (v_0, v_1, \dots, v_\ell)$, let $D_n = v_0 H_n + \dots + v_{\ell} H_{n-\ell}$. We call $D = \{D_n\}$ the \textit{derived sequence of $H$ with respect to $v$}. We study the question of what the gzd of $D_n$ looks like as $n \to \infty$. Our main result is Theorem \ref{thm_converge}, which essentially says that these gzds converge in a suitable sense. Additionally Theorem \ref{thm_converge} classifies three distinct types of behavior for this convergence. Theorem \ref{thm_converge} is closely related to the following result of \cite{GTNP} (which, interestingly, utilizes Theorem \ref{thm_fs} in its proof).

\begin{theorem} [Corollary 1 of \cite{GTNP}] \label{thm_grabner}
Suppose $H$ is a PLRS with weakly decreasing signature. Let $k \in \mathbb{N}$. Then the gzd of $k H_n$ is equal to $\sum_{i = n-a}^{n+b} r_i H_{n+i}$ for all $n$ with $a, \ b, \ r_i$ constants. Furthermore, $\sum_{i = -a}^b r_i \beta^i$ is the $\beta$-expansion of $k$. Thus in particular the number of summands in the gzd of $k H_n$ does not change as $n \to \infty$.
\end{theorem}

We remark briefly that using the terminology developed later in this paper (see Section \ref{section_prelim_2}), Theorem \ref{thm_grabner} says that if $\sigma$ is weakly decreasing, then the derived sequence of $H_\sigma$ with respect to $(k)$, call it $D$, is in Case 1 of Theorem \ref{thm_converge}, and thus $D$ converges in gzd to $r_b r_{b-1} \cdots r_0.r_{-1} \cdots r_{-a}$.

In Section \ref{section_prelim_2}, we start by discussing relevant properties of linear recurrence sequences, introduce derived sequences and extended representations, and discuss how the gzd is the output of a modified greedy algorithm. In Section \ref{section_convergence}, we discuss convergence in gzd and then state and prove Theorem \ref{thm_converge}. We believe Theorem \ref{thm_converge} is of interest in its own right independent of its use in proving Theorem \ref{thm_summand_minimal}. In Section \ref{section_summand_minimality}, we first prove using a simple trick that if $\sigma$ does not satisfy the Parry condition (see Definition \ref{defn_parry}) and $c_1 \neq 1$, then $H_\sigma$ is not summand minimal. We then show that if $\sigma$ does satisfy the Parry condition or $c_1 = 1$, and $H_\sigma$ is summand minimal, then $\sigma$ is weakly decreasing. The idea of the proof is to consider the derived sequence of $H_{\sigma}$ with respect to $(c_1+1)$. If we end up in Cases 2 or 3 of Theorem \ref{thm_converge}, then the gzd of $(c_1+1) H_n$ has arbitrarily many summands for $n$ large (and in particular more than $c_1+1$ summands), and thus $H_\sigma$ is not summand minimal. We then show that if $\sigma$ satisfies the Parry condition and we are in Case 1 of Theorem \ref{thm_converge} and $H_\sigma$ is summand minimal, then $\sigma$ is weakly decreasing. The proof utilizes a theorem of Brauer \cite{Br} that $x^t - c_1 x^{t-1} - \dots - c_t$ is irreducible in $\mathbb{Q}[x]$ if $c_1 \geq \dots \geq c_t \geq 1$ with $c_i \in \mathbb{Z}$. The case that $c_1 = 1$ instead utilizes a theorem of Schinzel \cite{Schi} regarding the factorization over $\mathbb{Q}[x]$ of the polynomials $x^r - 2 x^s + 1$.

%% file: preliminaries.tex
\section{Preliminaries} \label{section_prelim}

\subsection{The generalized Zeckendorf decomposition}
A \textit{(linear) recurrence sequence} is a sequence $H = \{H_n\}_{n = \ell}^\infty$ which satisfies a linear recurrence relation, i.e., there exist constants $a_0, a_1, \dots,  a_t \in \mathbb{C}$ such that $a_0 H_n + a_1 H_{n-1} + \dots + a_t H_{n-t} = 0$ for all $n \geq t+\ell$. In this case we say that $H$ satisfies the recurrence relation $(a_0, \dots, a_t)$. We call $\ell$ the \textit{minimal index of $H$}.

Suppose $H$ is a recurrence sequence with minimal index $-(t-1)$. We say $H$ has \textit{ideal initial conditions} if $H_{-(t-1)} = \dots = H_{-1} = 0$ and $H_0 = 1$. We shall be particularly interested in the case where $H$ has ideal initial conditions and $H_n = c_1 H_{n-1} + \dots + c_t H_{n-t}$ with $c_1, \ c_t \geq 1$ and $c_i \in \mathbb{N}_0$. In this case we call $\sigma = (c_1, \dots, c_t)$ the \textit{signature of $H$} and we say that $H$ is a \textit{positive linear recurrence sequence (PLRS)} (except when $H_n = H_{n-1}$ with $H_0 = 1$, which we do not consider to be a valid PLRS). A PLRS with signature $\sigma$ will be denoted by $H_\sigma$.

\begin{definition}\label{defn_representation}
Suppose $\rho = (r_n, r_{n-1}, \dots, r_\ell)$ is a tuple of non-negative integers. Let $\rho' = \sum_{i=\ell}^s r_i H_i$. We say that $\rho$ is a \textit{representation of $\rho'$ using $H$}.
\end{definition}

\begin{definition}\label{def:block}
Suppose $\sigma = (c_1, \dots, c_t)$. Then $(b_1, \dots, b_k)$ is called an \textit{allowable block} (or \textit{valid block}) if $k \leq t$, $b_i = c_i$ for $i < k$, and $0 \leq b_k < c_k$. We say that a representation is \textit{allowable} if it is composed of concatenated allowable blocks.
\end{definition}

\begin{example}
If $\sigma = (4, 3, 2)$ then the set of allowable blocks is
$$\{(0), (1), (2), (3), (4, 0), (4, 1), (4, 2), (4, 3, 0), (4, 3, 1)\}.$$
If $\sigma = (2, 0, 0, 3)$ then the allowable blocks are
$$\{(0), (1), (2, 0, 0, 0), (2, 0, 0, 1), (2, 0, 0, 2)\}.$$
\end{example}

\begin{theorem}[Generalized Zeckendorf theorem; Theorem 1.1 in \cite{HW} and Theorem 1.1 in \cite{MW}]\label{thm:gzd}
Given a PLRS, $H_\sigma$, every non-negative integer $m$ has a unique allowable representation using $H_\sigma$.
\end{theorem}

\begin{definition}
The representation as in Theorem \ref{thm:gzd} is called the \textit{generalized Zeckendorf decomposition (gzd)}. Given $m \in \mathbb{N}_0$, we let $\textnormal{GZD}(m)$ denote the gzd of $m$.
\end{definition}

\begin{example}
Suppose $\sigma = (1, 1)$, the signature for the Fibonacci numbers. The allowable blocks are $\{(0), (1, 0)\}$. Therefore, Theorem \ref{thm:gzd} implies that every integer has a unique representation composed of $(0)$ and $(1, 0)$, which is Zeckendorf's theorem.
\end{example}

The next definition is most often used in the context of $\beta$-expansions. It was first noted by Parry \cite{Pa} to be of great importance in the study of $\beta$-expansions.
\begin{definition} \label{defn_parry}
A string $(c_1, \dots, c_t)$ is said to satisfy the \textit{Parry condition} if all of the following hold:
\begin{equation}
\begin{split}
(c_1, c_2, \dots, c_t) &\ >_\mathrm{lex}\ (c_2, c_3, \dots, c_t, 0),\\
(c_1, c_2, \dots, c_t)    &\ >_\mathrm{lex}\ (c_3, \dots, c_t, 0, 0), \\
    &\ \ \ \vdots \\
(c_1, c_2, \dots, c_t)    & \ >_\mathrm{lex}\ (c_t, 0, \dots, 0). \label{eqn_parry}
\end{split}
\end{equation}
\end{definition}

\begin{remark} \label{remark_parry_block}
Prior to the work of Hamlin and Webb \cite{HW}, and Miller and Wang \cite{MW}, the generalized Zeckendorf theorem was only known in the case where $\sigma$ satisfies the Parry condition (in particular, when $\sigma$ is weakly decreasing). In that case, a representation $\rho = (r_n, \dots, r_0)$ is allowable if and only if $(r_{i-1}, \dots, r_{i-t}) <_\mathrm{lex} (c_1, \dots, c_t) = \sigma$ for all $i$ (see, e.g., \cite{GTNP}).
\end{remark}

\subsection{Borrowing and carrying}

Suppose $H = H_\sigma$ with $\sigma = (c_1, \dots, c_t)$. If $t \geq 2$, then $H$ contains zeros, namely $H_{-1}, \dots, H_{-(t-1)}$. Let $\rho = (r_n, \dots, r_{-(t-1)})$. Clearly $\sum_{i = -(t-1)}^n r_s H_s = \sum_{i = 0}^n r_s H_s$. Thus any $r_i$ with $-(t-1) \leq i \leq -1$ may be changed to an arbitrarily large number without having a meaningful effect on $\rho$ as a representation of $\rho'$. Because of this we shall always assume that any representation has all of its negative index entries equal to $\infty$. We use the shorthand $\infty_{t-1}$ to denote $\underbrace{\infty, \dots, \infty}_{t-1}$. The utility of using these infinities shall become clear momentarily.

The signature provides a way of moving between representations of the same number. For example, if $\sigma = (10)$, one representation for $312$ is $(3, 1, 2)$ (indeed, this is the gzd). However, we may also represent $312$ as $(2, 11, 2)$ (by ``borrowing'' from the 100's place). Analogously, say we have the representation $(6, 23, 4)$. Since the 10's place currently has $23 \geq 10$, we can ``carry'' to the 100's place to get the representation $(7, 13, 4)$, and then carry again to get $(8, 3, 4)$.

The ideas of ``borrowing'' and ``carrying'' from base-$d$ arithmetic extend to all recurrences. For example, suppose $\sigma = (2, 1)$. Then $H_\sigma = \{0, 1, 2, 5, 12, \dots\}$. Given the representation $(3, 0, 0, \infty)$ (which represents $15 = 3 \cdot 5$), we can ``borrow'' from index $2$ to get the representation $(2, 2, 1, \infty)$ (which still represents $15 = 2 \cdot 5 + 2 \cdot 2 + 1 \cdot 1$). If we subsequently borrow from index $1$, we get $(2, 1, 3, \infty+1)$. If we extend our arithmetic to include $\infty$ such that $\infty \pm n = \infty$ for any $n < \infty$, and $\infty \cdot 0 = 0$, then we can still ``borrow'' even when it results in terms accumulating in the ``infinities places.'' 

Suppose now we have the representation $(3, 4, \infty)$ (with $\sigma = (2, 1)$ as before). Since $3 \geq 2 = c_1$ and $4 \geq 1 = c_2$, we can ``carry'' to index 2 to get $(1, 1, 3, \infty)$. We can carry again to index 1 to get the representation $(1, 2, 1, \infty-1) = (1, 2, 1, \infty)$. Thus, using $\infty$ also allows us to ``carry'' even when the carry operation utilizes the infinities places.

\begin{definition}\label{def:carryborrow} 
Let $\rho = (r_n, r_{n-1}, \dots, r_0,\infty_{t-1})$ be a representation using $H_\sigma$ with $\sigma = (c_1, \dots, c_t)$. Let $\mathcal{B}(\rho, i)$ and $\mathcal{C}(\rho, i)$ be defined as
\begin{gather*}
\mathcal{B}(\rho, i)\ :=\ (r_n, \dots, r_i - 1, r_{i-1}+c_1, \dots, r_{i-t} + c_t, r_{i-(t+1)}, \dots, r_0, \infty_{t-1}), \\
\mathcal{C}(\rho, i)\ :=\ (r_n, \dots, r_i+1, r_{i-1} - c_1, \dots, r_{i - t} - c_t, r_{i - (t+1)}, \dots, r_0,\infty_{t-1}).
\end{gather*}
We call the application of $\mathcal{B}(\rho, i)$ \textit{borrowing from $i$} and the application of $\mathcal{C}(\rho, i)$ \textit{carrying to $i$} (see Tables \ref{table_borrow} and \ref{table_carry}).
\end{definition}
\vspace{-0.25cm}
\begin{table}[H]
\centering
\resizebox{\textwidth}{!}{
    \begin{tabular}{| c || c | c | c | c | c | c | c | c | c | c | c | c |}
    \hline
    Remark ($\downarrow$)/Index ($\to$)       & ${n}$ & $\cdots$ & ${i}$ & ${i-1}$  & $\cdots$ & ${i-t}$  & ${i-(t+1)}$ & $\cdots$ & ${0}$ & ${-1}$ & $\cdots$ & ${-(t-1)}$ \\ \hline \hline
    ${\rho}$ & $r_n$        & $\cdots$ & $r_i$        & $r_{i-1}$       & $\cdots$ & $r_{i-t}$       & $r_{i-(t+1)}$      & $\cdots$ & $r_0$        & $\infty$      & $\cdots$ & $\infty$\\
    Borrow from $i$.       & ~            & ~       & $-1$         & $c_1$           & $\cdots$ & $c_t$           & ~                  & ~       & ~            & ~             & ~       & ~\\ \hline
    ${\mathcal{B}(\rho, i)}$ & $r_n$        & $\cdots$ & $r_i - 1$    & $r_{i-1} + c_1$ & $\cdots$ & $r_{i-t} + c_t$ & $r_{i - (t+1)}$    & $\cdots$ & $r_0$        & $\infty$      & $\cdots$ & $\infty$\\ \hline
    \end{tabular}
    }
\caption{Borrow from $i$.} \label{table_borrow}
\end{table}
\vspace{-0.25cm}
\begin{table}[H]
\centering
\resizebox{\textwidth}{!}{
    \begin{tabular}{| c || c | c | c | c | c | c | c | c | c | c | c | c |}
    \hline
    Remark ($\downarrow$)/Index ($\to$)  & ${n}$ & $\cdots$ & ${i}$ & ${i-1}$  & $\cdots$ & ${i-t}$  & ${i-(t+1)}$ & $\cdots$ & ${0}$ & ${-1}$ & $\cdots$ &    ${-(t-1)}$ \\ \hline \hline
    ${\rho}$ & $r_n$        & $\cdots$ & $r_i$        & $r_{i-1}$       & $\cdots$ & $r_{i-t}$       & $r_{i-(t+1)}$      & $\cdots$ & $r_0$        & $\infty$      & $\cdots$   &   $\infty$ \\
    Carry to $i$.	& ~            & ~       & $1$          & $-c_1$          & $\cdots$ & $-c_t$          & ~                  & ~       & ~            & ~             & ~ &  ~     \\ \hline
    ${\mathcal{C}(\rho, i)}$ & $r_n$        & $\cdots$ & $r_i + 1$    & $r_{i-1} - c_1$ & $\cdots$ & $r_{i-t} - c_t$ & $r_{i - (t+1)}$    & $\cdots$ & $r_0$        & $\infty$      & $\cdots$   &   $\infty$ \\ \hline
    \end{tabular}
    }
\caption{Carry to $i$.} \label{table_carry}
\end{table}

\begin{definition}
Let $\rho = (r_n, \dots, r_0, \infty_{t-1})$ be a representation using $H_\sigma$. We say that we are \textit{able to carry to $i$} if $r_{i-j} \geq c_j$ for all $1 \leq j \leq t$, i.e., $(r_{i-1},\dots,r_{i-t})\geq (c_1,\dots,c_t)$ pointwise.
\end{definition}

\begin{remark}\label{remark_sum_bor_car}
When we borrow from/carry to an index $i\geq t$, the change in the number of summands is $\pm(-1+c_1+\dots+c_t)$. When we borrow from/carry to an index $i<t$, the change in the number of summands is $\pm(-1+c_1+\dots+c_i)$. Notice that if we borrow from an index $i$ and then carry to an index $j$ with $j > i$, then the net change in the number of summands is non-positive.
\end{remark}

%% file: weakly_decreasing.tex
\section{Weakly decreasing signature implies summand minimality} \label{section_weakly_decreasing}

\begin{proposition} \label{prop_weakly_decreasing}
Suppose $\sigma$ is weakly decreasing. Then $H_\sigma$ is summand minimal.
\end{proposition}
\begin{proof}
Suppose $\sigma$ is weakly decreasing. Let
\begin{equation*}
\rho = \rho_1 = (r_n, \dots, r_0, \infty_{t-1})
\end{equation*}
be a non-allowable representation using $H_\sigma$. Let $i_1$ be the largest index such that 
\begin{equation} \label{eqn_rho_big}
(r_{i_1-1}, \dots, r_{i_1 - t}) \geq_\mathrm{lex} (c_1, \dots, c_t),
\end{equation}
which must exist by Remark \ref{remark_parry_block}. If we are able to carry to $i_1$, then we let $\rho_2 = \mathcal{C}(\rho_1, i_1)$. Then $\rho_2$ clearly has at most as many summands as $\rho_1$. 

If we are not able to carry to $i_1$, then we find the largest $j_1$ (necessarily less than $t$) such that
\begin{gather*}
\begin{split}
r_{i_1-1} &\geq c_1, \\
&\vdots \\
r_{i_1 - (j_1-1)} &\geq c_{j_1-1}, \\
r_{i_1 - j_1} &> c_{j_1}.
\end{split}
\end{gather*}
We then borrow from $i_1-j_1$ to obtain $\nu_1 = \mathcal{B}(\rho_1, i_1-j_1) = (s_n, \dots, s_0, \infty_{t-1})$. Notice that
\begin{align}
s_{i_1-1} = r_{i_1-1} &\geq c_1, \notag \\
& \vdots \notag \\
s_{i_1 - (j_1-1)} &\geq c_{j_1-1}, \notag \\
s_{i_1 - j_1} = r_{i_1 - j_1} - 1 &\geq c_{j_1}, \notag \\
s_{i_1 - (j_1+1)} = r_{i_1 - (j_1+1)} + c_1 \geq c_1 &\geq c_{j_1+1}, \ \ \ \ \ \ \ \ \ \ \ \ \ \ \ \ \ \ \ \ \ \ \ \ \ \ \ \ \ \label{eqn_weakly_first}\\
& \vdots \\
s_{i_1 - t} = r_{i_1 - t} + c_{t-j_1} \geq c_{t-j_1} &\geq c_t. \label{eqn_weakly_last}
\end{align}
Thus we are able to carry to $i_1$ (the fact that $\sigma$ is weakly decreasing is used in Equations \eqref{eqn_weakly_first}-\eqref{eqn_weakly_last}). Let $\rho_2 = \mathcal{C}(\nu_1, i_1)$. Since $i_1 > i_1-j_1$ (as $1 \leq j_1 \leq t$), the number of summands in $\rho_2$ is less than or equal to the number of summands in $\rho_1$ by Remark \ref{remark_sum_bor_car}.

We repeat the above procedure to find $i_2$ such that Equation \eqref{eqn_rho_big} holds, with $i_1$ replaced by $i_2$, and $r_k$ now representing the entries of $\rho_2$, etc. If this procedure terminates, which is to say that at the $\ell$th iteration such an $i_\ell$ is not found, then $\rho_\ell$ is allowable and hence is the gzd of $\rho'$ (see Remark \ref{remark_parry_block}). Since the number of summands never increases, we conclude that $H_\sigma$ is summand minimal. The proof that this procedure terminates can be found in Appendix \ref{section_algorithm} (as well as a description and proof of termination of a more general algorithm for moving from any representation to the gzd for all signatures).
\end{proof}

\begin{example}
To illustrate the proof of Proposition \ref{prop_weakly_decreasing}, we consider the example with $\sigma = (3,2,1,1)$ and $\rho= (3,2,1,0,4,3,0,\infty_3)$.

\begin{table}[H]
\resizebox{\textwidth}{!}{
\begin{tabular}{|L||c|c|c|c|c|c|c|c|c|c|c||c|}
\hline
Remark ($\downarrow$)/Index  ($\to$)        & $7$ & $6$ & $5$ & $4$ & $3$ & $2$ & $1$ & $0$ & $-1$     & $-2$     & $-3$     & Representation  \\
\hline \hline
$(4,3,0,\infty)>_\mathrm{lex} (3,2,1,1)$, so $i_1=3$. We cannot carry as $(4,3,0,\infty) \not\geq \sigma$ pointwise. &     & $3$ & $2$ & $1$ & $0$ & $4$ & $3$ & $0$ & $\infty$ & $\infty$ & $\infty$ & $\rho=\rho_1$  \\
Since $r_{3-1}=4 \geq c_1 = 3$ and $r_{3-2}=3>2=c_2$, we have $j_1=2$. We thus borrow from $i_1-j_1=1$. &  &     &     &     &     &     & -1    & 3    &  2    & 1          &  1          &  \\
\hline
We can carry to $i_1=3$ as $(3,2,3,\infty)\geq \sigma$ pointwise. &  &   3   &   2   &  1    &   0   &   4 &   2   &   3   &   $\infty$    &   $\infty$    &   $\infty$    &  $\nu_1=\mathcal B (\rho_1, 1)$ \\
 &    &   &   &   & 1  & -3  & -2 &  -1 & -1  &   &   &  \\
\hline
$(3,2,1,1) \geq_{\mathrm{lex}} \sigma$, so $i_2=7$. &  &   3   &   2   &  1    &   1   &   1 &   0   &   2   &   $\infty$    &   $\infty$    &   $\infty$    & $\rho_2=\mathcal C(\nu_1,3)$   \\

We can carry immediately to $i_2$. & 1  & -3    &  -2    & -1  & -1    &     &     &    &     &          &           &  \\
\hline
$i_3$ does not exist, so $\rho_3$ is the gzd of $\rho'$. & 1 &   0   &   0   &  0    &   0   &   1 &   0   &   2   &   $\infty$    &   $\infty$    &   $\infty$    & $\rho_3=\mathcal C(\rho_2,7)$  \\ \hline
\end{tabular}
}
\end{table}
\noindent Note that $13 = \# \textnormal{ of summands of }\rho_1 \geq \# \textnormal{ of summands of }\rho_2 \geq \# \textnormal{ of summands of }\rho_3 = 4$.
\end{example}

We remark that the the procedure used in the proof of Proposition \ref{prop_weakly_decreasing} is similar to the algorithm used by Frougny and Solomyak in \cite{FS} to prove Theorem \ref{thm_fs}.

%% file: preliminaries_2.tex
\section{Preliminaries, round two} \label{section_prelim_2}

We now turn our attention towards proving Theorem \ref{thm_converge}, which says in essence that certain natural sequences of numbers (so-called derived sequences; see Definition \ref{defn_derived}) have convergent gzds. To make this notion precise, we must review some fundamental results about recurrence sequences, and introduce some new notation and terminology.

\subsection{Recurrence sequences}

Suppose $H = \{H_n\}_{n = 0}^\infty$ is a recurrence sequence (not necessarily a PLRS). For each recurrence relation $a = (a_0, \dots, a_t)$ that $H$ satisfies, we have an associated polynomial $\mathcal{P}_a(x) := a_0 x^t + a_1 x^{t-1} + \dots + a_t$. Let 
$$\mathcal{I}(H)\ := \ \{\mathcal{P}_a \in \mathbb{C}[x] : H \textnormal{ satisfies } a \}.$$ 
It is immediate to see that $\mathcal{I}(H)$ is an ideal in $\mathbb{C}[x]$. Since $\mathbb{C}[x]$ is a principal ideal domain, there is a unique monic polynomial of minimal degree which generates $\mathcal{I}(H)$. We call this polynomial the \textit{minimal polynomial of $H$}. The following is almost certainly a known result, though we were unable to find a clear reference in the literature. We prove it in Appendix \ref{appendix_min_poly} for completeness.

\begin{proposition}\label{prop_min_poly}
Suppose $H$ is a recurrence sequence generated by the recurrence relation $a = (1, -a_1, \dots, -a_t)$ with ideal initial conditions. Then $\mathcal{P}_a$ is the minimal polynomial of $H$.
\end{proposition}

\noindent In particular, any PLRS has $\mathcal{P}_{(1, -c_1, \dots, -c_t)}$ as its minimal polynomial.

One of the most fundamental results about recurrence sequences is the following (for a proof, see Corollary 2.24 in \cite{elaydi}).

\begin{theorem}\label{thm_growth}
Let $a = (a_0, a_1, \dots, a_t)$ be a recurrence relation. Let $\lambda_1, \dots, \lambda_s$ be the distinct roots of $\mathcal{P}_a$. Suppose $\lambda_i$ is a root of multiplicity $m_i$. Then a sequence $H$ satisfies $a$ if and only if 
\begin{gather}
H_n\ =\  p_1(n) \lambda_1^n + \dots + p_s(n) \lambda_s^n, \label{eqn_sum_roots}
\end{gather}
where $p_i(n)$ is a polynomial of degree at most $m_i - 1$. 
\end{theorem}

Given a polynomial $f$, we say that a root $\beta$ of $f$ \textit{dominates} or is \textit{dominating} if $|\beta| \geq |\lambda_i|$ for all roots $\lambda_i$ of $f$. We say that a monic polynomial is of \textit{positive type} if it is of the form $f = x^t - c_1 x^{t-1} - \dots - c_t$ with $c_1, c_t \geq 1$ and $c_i \geq 0$ and $f \neq x-1$. Thus PLRSs have minimal polynomials of positive type by Proposition \ref{prop_min_poly}. The following result has been rederived several times. We include its proof in Appendix \ref{appendix_pos_type} for completeness.

\begin{proposition}\label{prop_pos_type}
Suppose $f$ is a polynomial of positive type. Then $f$ has a unique real positive dominating root $\beta$ of multiplicity one with $\beta > 1$. All other roots are non-positive.
\end{proposition}

\begin{corollary}\label{cor_beta_coeff}
Suppose $H$ is a positive linear recurrence sequence generated by ideal initial conditions and with minimal polynomial $f$. Let $\beta = \lambda_1$ be the dominating root. Then $p_1(n)$ as in Equation \eqref{eqn_sum_roots} is a positive constant, $c$.
\end{corollary}

\begin{proof}
Since $\beta$ is a simple root, the polynomial $p_1(n)$ must be degree 0, and hence is equal to some constant $c$. If $c = 0$, then by Theorem \ref{thm_growth}, the polynomial $f(x)/(x-\beta)$ is in $\mathcal{I}(H)$, contradicting the fact that $f(x)$ is the minimal polynomial. Since $\beta$ dominates and $c \neq 0$, past some point, $c \beta^n$ and $H_n$ always have the same sign. However, clearly $\beta^n$ is always positive, and $H_n$ is always positive, so $c > 0$.
\end{proof}

\begin{corollary}\label{cor_no_pos}
Suppose $f$ is a polynomial of positive type. Let $\beta$ be its dominating root. Then $f(x)/(x-\beta)$ has no positive dominating root.
\end{corollary}

Corollary \ref{cor_no_pos} will be used critically with the following result of Bell and Gerhold \cite{BG}.

\begin{theorem}[Theorem 2 of \cite{BG}]\label{thm_oscillate}
Suppose $H$ is a (real) non-zero recurrence sequence with no positive dominating root. Then the sets $\{n \in \mathbb{N} : H_n > 0\}$ and $\{n \in \mathbb{N} : H_n < 0\}$ are both infinite.
\end{theorem}

\noindent In particular, if $f$ is of positive type, then a recurrence sequence whose minimal polynomial divides $f/(x-\beta)$ will oscillate between positive and negative infinitely often. This will be crucial in the proof of Case 2 of Theorem \ref{thm_converge}.

In the sequel, the letter $\beta$ will always refer to the dominating root of a polynomial of positive type.

\subsection{Extended representations and derived sequences}

Suppose $v = (v_0, v_{-1}, \dots, v_{-n})$ is a finite string with integer entries. We let $[v]$ denote the map from $\mathbb{Z}$ to $\mathbb{Z}$ defined by $[v](i) = v_i$ for $-n \leq i \leq 0$, and $[v](i) = 0$ otherwise. Alternatively, we can think of $[v]$ as the $\mathbb{Z}$-indexed bi-infinite string with $v$ ``placed'' at the zero index and padded with zeros everywhere else. We shall make use of both perspectives.

\begin{definition}
We call a map $\gamma: \mathbb{Z} \to \mathbb{Z}$ an \textit{extended representation} if $\gamma$ has compact support in the discrete topology.
\end{definition}

If the support of $\gamma$ is contained in $\mathbb{N}_0$, as 
before (see Definition \ref{defn_representation}) we define $\gamma'$ as
\begin{equation*}
\gamma'\ :=\ \gamma(0) H_0 + \gamma(1) H_1 + \cdots.
\end{equation*}
We can think of $\gamma'$ as the number represented by $\gamma$ using the sequence $H$.

Let $\mathcal{E}$ denote the space of all extended representations. Let $S: \mathcal{E} \to \mathcal{E}$ denote the shift map on extended representations, that is
$$S^n(\gamma)(i) := \gamma(i-n).$$
If we instead think of $\gamma$ as a $\mathbb{Z}$-indexed bi-infinite string with indices increasing from right to left (as is the case for usual base-$d$ number systems), then $S^n(\gamma)$ is the result of shifting $\gamma$ to the left by $n$ places (in case $n$ is negative, then $S^n(\gamma)$ is the result of shifting $\gamma$ to the right by $-n$ places).

If $\gamma$ is an extended representation, then 
\begin{gather*}
\mathcal{L}(\gamma)\ :=\ \sup \textnormal{supp} \ \gamma, \\
\mathcal{R}(\gamma)\ :=\ \inf \textnormal{supp} \ \gamma,
\end{gather*}
where $\textnormal{supp} \ \gamma$ is the support of $\gamma$. At times we shall represent an extended representation as
\begin{equation*}
\gamma\ =\ \gamma(\mathcal{L}(\gamma)) \ \gamma(\mathcal{L}(\gamma)-1) \cdots \gamma(0) . \gamma(-1) \cdots \gamma(\mathcal{R}(\gamma)),
\end{equation*}
with the decimal point lying between the zeroeth and minus first positions. We will also use the notation $(x_1 \cdots x_j)^m$ to denote $x_1 \cdots x_j$ concatenated with itself $m$ times. We let $(x_1 \cdots x_j)^\omega$ denote $x_1 \cdots x_j$ concatenated with itself a countably infinite number of times.

Given a finite string $v = (v_0, \dots, v_n)$, we define the polynomial
\begin{equation*}
\mathcal{P}_v(x)\ :=\ v_0 x^n + \cdots + v_n.
\end{equation*}
Given an extended representation $\gamma$, we define the rational function, 
\begin{equation*}
\mathcal{Q}_\gamma(x)\ :=\ \sum \gamma(i) x^i.
\end{equation*}
Given $\gamma$, we may also define the polynomial
\begin{equation*}
\mathcal{P}_\gamma(x)\ :=\ x^{-\mathcal{R}(\gamma)} \mathcal{Q}_\gamma (x).
\end{equation*}

\begin{remark}\label{remark_p_q}
In the sequel we will often be interested in the sign of $\mathcal{P}_\gamma(\beta)$. Since $\beta > 0$, the signs of $\mathcal{P}_\gamma(\beta)$ and $\mathcal{Q}_\gamma(\beta)$ are always the same.
\end{remark}

\begin{definition} \label{defn_derived}
Given a finite string $v = (v_k, \dots, v_0)$ (resp., extended representation $\gamma$) and a recurrence sequence $H$, we define the \textit{derived sequence of $H$ with respect to $v$} (resp., $\gamma$), denoted $D = \textnormal{Der}(v, H)$ (resp., $D = \textnormal{Der}(\gamma, H)$), with 
\begin{gather*}
D_n\ := \ v_k H_{n} + v_{k-1} H_{n-1} + \dots + v_0 H_{n-k} \\
\textnormal{(resp., } D_n\ :=\ \sum_i \gamma(i) \ H_{n+i} \textnormal{).}
\end{gather*}
\end{definition}

Some terms of $D$ may be undefined; however, past some finite point, all subsequent terms of $D$ are defined. Because we will almost always only care about sufficiently large terms of $D$, we do not worry about this issue.

\begin{proposition}
Let $H$ be a recurrence sequence and let $D$ be the derived sequence of $H$ with respect to some finite string or extended representation. Then every recurrence relation satisfied by $H$ is also satisfied by $D$.
\end{proposition}

\begin{proof}
For simplicity, we focus on the derived sequence with respect to $v = (v_k, \dots, v_0)$. The proof for the case of the derived sequence with respect to an extended representation is virtually identical.

Suppose $H$ is a recurrence sequence with $f \in \mathcal{I}(H)$. Let the roots of $f$ be $\lambda_1, \dots, \lambda_s$. By Theorem \ref{thm_growth}, we can write $H_n$ as
$$H_n\ =\ \sum_{i = 1}^s p_i(n) \lambda_i^n. $$
We are interested in 
\begin{equation}
\begin{split} \label{eqn_eval_poly}
D_n &\ =\  v_k H_{n} + \dots + v_0 H_{n-k} \\
    &\ =\  \sum_{j = 0}^k v_j \bigg ( \sum_{i = 1}^s p_i (n-j) \lambda_i^{n-j} \bigg) \\
    &\ =\  \sum_j v_j \bigg (\sum_i p_i (n-j) \lambda_i^{-j} \lambda_i^n \bigg )\\
    &\ =\  \sum_i \bigg (\sum_j v_j \ p_i(n-j) \lambda_i^{-j} \bigg ) \lambda_i^n.
\end{split}
\end{equation}
Notice that $\sum_j v_j \ p_i(n-j) \lambda_i^{-j}$ is a polynomial in $n$ of degree at most equal to the degree of $p_i$. Thus, by Theorem \ref{thm_growth}, we conclude that $f \in \mathcal{I}(D)$.
\end{proof}

\begin{remark}\label{remark_derived_zero}
It is possible for $D$ to satisfy more recurrence relations than $H$. For example, if $v$ is some finite string and $H$ satisfies $v$, then $D$ is identically zero and thus satisfies every recurrence relation. In fact, we have that $D$ is identically zero if and only if $H$ satisfies $v$.
\end{remark}

\begin{remark}\label{remark_coeff}
Suppose $H$ is a PLRS with minimal polynomial $f$ and $D = \textnormal{Der}(v, H)$. Since $\beta$ is a simple root, the coefficient of $\beta^n$ in the last line of Equation \eqref{eqn_eval_poly} is $c \mathcal{P}_v(\beta)$. Since $c > 0$ by Corollary \ref{cor_beta_coeff}, we have that if $\mathcal{P}_v(\beta) > 0$, then $D$ is eventually always positive. If $\mathcal{P}_v(\beta) < 0$, then $D$ is eventually always negative. Lastly, if $\mathcal{P}_v(\beta) = 0$, then the minimal polynomial for $D$ divides $f/(x-\beta)$ and thus $D$ oscillates between positive and negative infinitely often by Corollary \ref{cor_no_pos} and Theorem \ref{thm_oscillate}.
\end{remark}

\subsection{Gzd as output of a greedy algorithm} \label{section_greedy}

In linear numeration systems, one studies the representation obtained by the usual greedy algorithm. However, the gzd is not always the output of this greedy algorithm. For example, suppose $H = H_{(1, 3)}$. Consider the derived sequence for $(2)$. The dominating root of the minimal polynomial, $x^2 - x - 3$, is $\beta \approx 2.303 > 2$. Since $H_{n+1}/H_n \to \beta$ as $n \to \infty$, when applying the above greedy algorithm to $2 H_n$ for $n$ sufficiently large, no term with index greater than $n$ is ever selected. Therefore the greedy algorithm selects two $H_n$ terms and then terminates. However $S^n([2])$ is not the gzd for $2 H_n$ since it is not composed of allowable blocks.

However, the gzd is the output of a different greedy algorithm, namely the one which greedily selects allowable blocks as noted in \cite{HW}. To formally describe this algorithm, we first set some notation. Suppose our signature is $(c_1, \dots, c_t)$. Let $A = c_1 + \dots + c_t$. The set of allowable blocks has a natural ordering, namely 
$$(0)\ <\ (1)\ <\ \cdots\ <\ (c_1 - 1)\ <\ (c_1, 0)\ <\ \cdots\ <\ (c_1, c_2, \dots, c_t - 1).$$ 
Let $v_i$ denote the $i$th smallest block with respect to this ordering (with $v_0 = (0)$ and $v_{A-1} = (c_1, c_2, \dots, c_t - 1)$). Let 
$$\gamma_{i, n}\ :=\ S^n([v_i]),$$
with $1 \leq i \leq A-1$ and $n \in \mathbb{Z}$. Let 
$$\Lambda \ :=\ \{\gamma_{i, n} : \textnormal{supp} \ \gamma_{i, n} \subset \mathbb{N}_0\}.$$
The elements of $\Lambda$ have a natural total ordering, namely $\gamma_{i_1, n_1} > \gamma_{i_2, n_2}$ iff $\gamma_{i_1, n_1}' > \gamma_{i_2, n_2}'$. Let $\Lambda_i$ denote the $i$th smallest element of $\Lambda$. One can easily verify (see \cite{HW}) that
\begin{gather}\label{eqn_gamma_order}
\cdots\ <\ \gamma_{1, n}\ <\ \gamma_{2, n} \ <\ \cdots\ <\ \gamma_{A-1, n} \ <\ \gamma_{1, n+1}\ <\ \cdots.
\end{gather}

Suppose $m \in \mathbb{N}$. We may obtain a representation of $m$ by first finding the largest index $i_1$ such that $\Lambda_{i_1}' \leq m$ but $\Lambda_{i_1 + 1}' > m$. We then repeat this process on $m - \Lambda_{i_1}'$ to obtain $\Lambda_{i_2}$, etc. This process terminates after finitely many steps. Let $k$ denote the number of steps involved.

\begin{proposition} [See Theorem 1.1 of \cite{HW}] \label{prop_greedy}
Let $m \in \mathbb{N}$. Then the representation $\Lambda_{i_1} + \Lambda_{i_2} + \dots + \Lambda_{i_k}$ is the gzd of $m$. In particular, the supports of $\Lambda_{i_j}$ and $\Lambda_{i_\ell}$ do not intersect unless $j = \ell$.
\end{proposition}

In the sequel, ``the greedy algorithm'' will always refer to the greedy algorithm of Proposition \ref{prop_greedy}.

Let $\Gamma$ denote the set of all $\gamma_{i, n}$. We can place a total ordering on the elements in $\Gamma$ using Equation \eqref{eqn_gamma_order}. Let $\Gamma_0$ (arbitrarily) be defined as $\gamma_{1, 0}$ and let $\Gamma_n$ with $n \in \mathbb{Z}$ denote the element in $\Gamma_n$ which is $n$ terms larger than $\Gamma_0$.

%% file: convergence.tex
\section{Convergence of gzds of a derived sequence} \label{section_convergence}

Given a PLRS $H$, what is the relationship between the gzds of elements in $\textnormal{Der}(v, H)$ where $v$ is some finite string? Briefly, Theorem \ref{thm_converge} says that this sequence of gzds converges in an exact sense with one pseudo-exception. We now define exactly what we mean by convergence.

\begin{definition}
Let $D = \textnormal{Der}(v, H)$. We define the \textit{normalized gzd} of $D_n$ as 
$$\textnormal{NGZD}(D_n)\ :=\ S^{-n}(\textnormal{GZD}(D_n)).$$
(Note that the normalized gzd takes as input not just an integer, but an integer and the index at which that integer occurs in some sequence.) We say that $D$ \textit{converges to $\alpha$ in gzd} if $\textnormal{NGZD}(D_n)$ converges to $\alpha$ pointwise (when viewed as maps $\mathbb{Z} \to \mathbb{Z}$ with the discrete topology). 
\end{definition} 
We extend the term \emph{allowable representation} to mean any extended representation of the form $\Gamma_{i_1} + \Gamma_{i_2} + \cdots$ such that $\textnormal{supp} \ \Gamma_{i_j} \cap \ \textnormal{supp} \ \Gamma_{i_k} = \emptyset$ unless $j = k$. We shall in general assume that allowable representations have finitely many terms (those with infinitely many terms will be called infinite allowable representations). Any time we write an allowable representation as $\Theta_1 + \Theta_2 + \cdots$ with $\Theta_j = \Gamma_{i_j}$ for some $i_j$, we shall assume that $\mathcal{L}(\Theta_k) > \mathcal{R}(\Theta_\ell)$ whenever $k > \ell$.

\begin{example}
Let $H$ be the Fibonacci sequence, and let $v = (2)$. Since $2 H_n = H_{n+1} + H_{n-2}$, we have that $\textnormal{NGZD}(D_n) = 10.01$ for all $n \geq 2$. Thus clearly $D$ converges to $10.01$ in gzd.
\end{example}

We now state our main theorem regarding convergence of gzds.

\begin{theorem}\label{thm_converge}
Suppose $H$ is a PLRS with signature $(c_1, \dots, c_t)$ and minimal polynomial $f$. Let $v$ be a finite string with $\mathcal{P}_v(\beta) > 0$. Let $D = \textnormal{Der}(v, H)$. Exactly one of the following happens:
\begin{enumerate}
\item There exists an allowable representation $\alpha$ such that $\mathcal{P}_{[v]-\alpha}$ is a multiple of $f$. Then $D$ converges to $\alpha$ in gzd.
\item There exists an allowable representation $\alpha = S^k([a_1, a_2, \dots, a_\ell])$ with $a_\ell \neq 0$ such that $\mathcal{P}_{[v]-\alpha}(\beta) = 0$, but $\mathcal{P}_{[v]-\alpha}$ is not a multiple of $f$. Let $K = \textnormal{Der}([v]-\alpha, H)$. Define 
\begin{gather*}
D_+\ =\ \{D_n : K_n \geq 0\}, \\
D_- \ =\ \{D_n : K_n < 0 \}.
\end{gather*}
Then the sets $D_+$ and $D_-$ are both infinite, $D_+$ converges to $\alpha$ in gzd, and $D_-$ converges to $S^k([a_1, a_2, \dots, a_{\ell-1}, (a_\ell-1), (c_1, c_2, \dots, (c_t - 1))^\omega])$ in gzd.

\item There does not exist an allowable representation $\alpha$ such that $P_{[v]-\alpha}(\beta) = 0$. Then $D$ converges in gzd to an infinite allowable representation.
\end{enumerate}
\end{theorem}

Given a signature $\sigma$ (or a PLRS $H_\sigma$) and a finite string $v$ with $\mathcal{P}_v(\beta) > 0$, we may refer to ``being in'' Case 1, 2, or 3 of Theorem \ref{thm_converge}, by which we mean that the hypotheses and conclusions of the respective case apply.

Before proving Theorem \ref{thm_converge}, we give examples of each of the three cases.

\begin{example}
Suppose $H$ has signature $(4, 2, 1)$. Let $v = (5, 1, 1, 3)$ and $D = \textnormal{Der}(v, H)$. Then $D$ converges to $\gamma = 10.324$ in gzd. Notice that
\begin{equation*}
\begin{split}
P_{[v]-\gamma} &\ =\ (5x^3 + x^2 + x + 3) - (x^4 + 3 x^2 + 2 x + 4) \\
    & \ =\ -x^4 + 5 x^3 - 2 x^2 - x - 1 \\
    &\ =\ (x^3 - 4 x^2 - 2 x - 1)(-x+1).
\end{split}
\end{equation*}
Thus $\textnormal{P}_{[v]-\alpha}$ is a multiple of the minimal polynomial. Properly interpreted, Theorem \ref{thm_fs} of Frougny and Solomyak \cite{FS} implies that if the signature is weakly decreasing, then given any finite string $v$ with $P_v(\beta) > 0$, we are always in Case 1 of Theorem \ref{thm_converge}. The special case when $v = (k)$ is the content of Theorem \ref{thm_grabner} of \cite{GTNP}.
\end{example}

\begin{example}
We present an example of Case 2, which is the most subtle of the cases. The following example is adapted from Example 1 of \cite{FS}. Let $\sigma = (2, 1, 0, 2)$. Let $v = (3, 0, 2)$. Below is a table showing the normalized gzd of some elements of $D$.
\begin{center}
\begin{tabular}{| l | l |}
\hline
$n$ & \textnormal{Normalized gzd of $D_n$} \\
\hline
$6$ & $10.2001$ \\
$7$ & $10.12101$ \\
$8$ & $10.200001$ \\
$9$ & $10.121012$ \\
$10$ & $10.20000001$ \\
$11$ & $10.121012101$ \\
\hline
\end{tabular}
\end{center}

The even index elements in $D$ converge in gzd to $10.2$, and the odd index elements converge in gzd to $10.1(2101)^\omega$. Note that $10.1(2101)^\omega$ is the allowable representation which is ``just less than'' $10.2$ (entirely analogously to how in the usual base-10 system, $0.9^\omega$ is the representation ``just less than'' $1.0$). The minimal polynomial of $H$ is $x^4 - 2 x^3 - x^2 - 2$, which factors over $\mathbb{Z}[x]$ as $(x^3 - 3x^2 + 2x - 2)(x+1)$. Notice that $\beta$ is a root of $x^3 - 3x^2 + 2x - 2$. Let $\alpha = S^1([1, 0, 2]) = 10.2$. Then $P_{[v]-\alpha} = -x^3 + 3x^2 - 2x + 2$, which has $\beta$ as a root but is not a multiple of the minimal polynomial of $H$.
\end{example}

\begin{example}
The following example is adapted from Example 2 of \cite{FS}. Let $\sigma = (2, 0, 1, 1)$. Let $v = [3]$. No $\alpha$ exists such that $P_{[v]-\alpha}(\beta) = 0$. Furthermore, $D$ converges in gzd to $10.111(00021)^\omega$.

\end{example}
In preparation for the proof of Theorem \ref{thm_converge}, we prove several simple ancillary lemmas.

\begin{lemma}\label{lemma_minus_adjacent}
For all $n \in \mathbb{Z}$, we have $\mathcal{Q}_{\Gamma_{n+1}-\Gamma_n}(\beta) = \beta^{\mathcal{R}(\Gamma_n)}$.
\end{lemma}
\begin{proof}
If $\Gamma_{n+1} \neq S^s([1])$ for some $s$, then the result is obvious. If $\Gamma_{n+1} = S^s([1])$, then the result is also obvious since $\beta^t - (c_1 \beta^{t-1} + \dots + (c_t-1)) = 1$. 
\end{proof}

\begin{lemma}\label{lemma_first_allowable}
Assume $n > m$.
\begin{enumerate}
\item Then $\mathcal{Q}_{\Gamma_n} (\beta) > \mathcal{Q}_{\Gamma_m} (\beta)$. 
\item We have $(S^k(\Gamma_n))' > (S^k(\Gamma_m))'$ for all $k$ for which both sides of the inequality are defined.
\item If $\mathcal{P}_v (\beta) > 0$, then there exists a unique $k \in \mathbb{Z}$ such that for all $\ell \leq k$, we have $\mathcal{P}_{[v] - \Gamma_{\ell}}(\beta) \geq 0$, and for all $\ell > k$, we have $\mathcal{P}_{[v] - \Gamma_{\ell}}(\beta) < 0$.
\end{enumerate}
\end{lemma}

\begin{proof}
(1) follows immediately from Lemma \ref{lemma_minus_adjacent}. (2) is simply a restatement of Equation \eqref{eqn_gamma_order}. (3) follows almost immediately from (1). First, note that as $n \to -\infty$, we have $\mathcal{Q}_{\Gamma_n}(\beta) \to 0$, and as $n \to \infty$, we have $\mathcal{Q}_{\Gamma_n}(\beta) \to \infty$. Thus by (1) and the preceding observations, we have that $\mathcal{Q}_{[v]-\Gamma_n}(\beta)$ monotonically goes to $\mathcal{Q}_{[v]}(\beta) > 0$ as $n \to -\infty$ and to $-\infty$ as $n \to \infty$. Therefore, utilizing Remark \ref{remark_p_q}, there must be a unique $k$ such that for all $\ell \leq k$, we have $\mathcal{P}_{[v]-\Gamma_{\ell}}(\beta) \geq 0$, and for all $\ell > k$, we have $\mathcal{P}_{[v]-\Gamma_{\ell}}(\beta) < 0$.
\end{proof}

\begin{remark}\label{remark_chosen_block}
Let $v$ be such that $\mathcal{P}_v(\beta) > 0$. Suppose $\Gamma_k \in \Gamma$. Let $E = \textnormal{Der}([v]-\Gamma_k, H)$ and $F = \textnormal{Der}([v]-\Gamma_{k+1}, H)$. If $E_s \geq 0$ and $F_s < 0$, then since the gzd is the output of the greedy algorithm of Proposition \ref{prop_greedy}, it follows from (2) of Lemma \ref{lemma_first_allowable} that $\Gamma_k$ is the first block of $\textnormal{NGZD}(D_s)$.
\end{remark}

We call $\Gamma_k$ as in (3) of Lemma \ref{lemma_first_allowable} the \textit{first candidate block of $[v]$}. We shall denote the first candidate block by $\Omega_1$. If $\mathcal{P}_{[v]-\Omega_1}(\beta) > 0$, then we may apply Lemma \ref{lemma_first_allowable} to $[v]-\Omega_1$ to obtain some block $\Omega_2$ which we call the second candidate block. We may repeatedly apply Lemma \ref{lemma_first_allowable} to obtain the \textit{$k$th candidate block}, $\Omega_k$, as long as $\mathcal{P}_{[v]-(\Omega_1+\dots+\Omega_{k-1})}(\beta) \neq 0$. If $\mathcal{P}_{[v]-(\Omega_1+\dots+\Omega_k)}(\beta) \neq 0$, we say that the \textit{asymptotic greedy algorithm continues}. If $\mathcal{P}_{[v]-(\Omega_1 + \dots + \Omega_k)}(\beta) = 0$, we call $\Omega_1 + \dots + \Omega_k$ the \textit{output of the asymptotic greedy algorithm for $[v]$}. If the asymptotic greedy algorithm always continues, we instead call $\sum_{j  = 1}^\infty \Omega_j$ the output of the asymptotic greedy algorithm for $[v]$.

\begin{remark}\label{remark_asymptotic}
Suppose $\Omega_1$ is the first candidate block of $[v]$. If $\mathcal{P}_{[v]-\Omega_1}(\beta) \neq 0$, then by Remark \ref{remark_coeff}, there exists an $N$ such that for $s > N$, $E_s$ is always positive and $F_s$ is always negative. Then by Remark \ref{remark_chosen_block}, $\Omega_1$ is the first block of $\textnormal{NGZD}(D_s)$. Repeated application of this argument shows that if $\Omega_1 + \dots + \Omega_k$ is the output of the asymptotic greedy algorithm, then $\Omega_1 + \dots + \Omega_{k-1}$ are the first $k-1$ blocks of $\textnormal{NGZD}(D_s)$ for $s$ sufficiently large.
\end{remark}

\begin{lemma}\label{lemma_repeating}
Let $\beta^t = c_1 \beta^{t-1} + c_2 \beta^{t-2} + \dots + c_t$. Let $\gamma = 0.(c_1 c_2 \cdots (c_t-1))^\omega$. Then $Q_\gamma(\beta) = 1$. In particular, if $\delta = S^k([1,(-c_1,-c_2\dots,-c_t+1)^n])$ for any finite $n$ and $k$, then $Q_\delta (\beta) > 0$.
\end{lemma}
\begin{proof}
We have $\beta^t - 1 = c_1 \beta^{t-1} + \dots + (c_t-1)$. Recall from Proposition \ref{prop_pos_type} that $\beta > 1$. Therefore,
\begin{equation*}
\begin{split}
Q_\gamma (\beta) &\ =\  (c_1\beta^{t-1} + \dots + (c_t-1)) (\beta^{-t}+\beta^{-2t} + \cdots) \\
    &\ =\  (\beta^t-1) \frac{\beta^{-t}}{1-\beta^{-t}}  \\
    & \ =\ 1.
\end{split}
\end{equation*}
\end{proof}

\begin{lemma}\label{lemma_one_biggest}
Suppose $\gamma$ is an allowable representation with support in $\mathbb{Z}\setminus\mathbb{N}_0$. Then $Q_{\gamma}(\beta) < 1$.
\end{lemma}
\begin{proof}
Suppose $\varepsilon = 1.0$. We shall modify $\varepsilon$ using borrows to end up with an extended representation $\delta$ such that $Q_\delta(\beta) = 1$ and $Q_\delta(i) \geq Q_\gamma(i)$ for all $i$, with strict inequality for at least one $i$. Suppose $\gamma = \Theta_1+\dots+\Theta_k$. Since $c_1 \geq 1$, by borrowing once from each of the places between $0$ and $\mathcal{L}(\Theta_1)+1$, we end up with an extended representation, $\eta$, with $\eta(\mathcal{L}(\Theta_1)+1) \geq 1$. Then, by borrowing from $\mathcal{L}(\Theta_1)+1$, we end up with an extended representation, $\nu$, which pointwise dominates $\Theta_1$ with strict dominance in at least one index (e.g., if $\Theta_1 = S^k([c_1, \dots, c_{r-1}, \ell])$ with $\ell < c_{r}$, then $\nu(k-r) > \ell = \Theta_1(k-r)$). Inductively, we can carry out the above procedure to get an extended representation which pointwise dominates $\gamma$ and has strict inequality in at least one index.
\end{proof}

\begin{corollary}\label{cor_greedy}
Suppose $\alpha = \Theta_1+\dots+\Theta_k$ is an allowable representation with $\Theta_1 = \Gamma_n$. Then $\mathcal{Q}_{\Gamma_{n+1}-\alpha}(\beta) > 0$.
\end{corollary}

\begin{proof}
By Lemmas \ref{lemma_minus_adjacent} and \ref{lemma_one_biggest},
\begin{equation*}
\begin{split}
\mathcal{Q}_{\Gamma_{n+1} - \alpha}(\beta) &\ =\  \mathcal{Q}_{(\Gamma_{n+1}-\Theta_1)-(\Theta_2+\dots+\Theta_k)}(\beta) \\
    &\ =\  \mathcal{Q}_{\Gamma_{n+1}-\Gamma_n}(\beta) - \mathcal{Q}_{\Theta_2+\dots+\Theta_k}(\beta) \\
    &\ > \ \beta^{\mathcal{R}(\Theta_1)} - \beta^{\mathcal{R}(\Theta_1)} \\
    &\ =\ 0.
\end{split}
\end{equation*}
\end{proof}

\begin{lemma}\label{lemma_greedy_true}
Suppose there exists an allowable sequence $\alpha = \Theta_1 + \dots + \Theta_k$ such that $\mathcal{P}_{[v]-\alpha}(\beta) = 0$. Then $\alpha$ is the output of the asymptotic greedy algorithm for $[v]$.
\end{lemma}
\begin{proof}
We first show that $\Theta_1$ is the first candidate block for $[v]$. Suppose $\Theta_1 = \Gamma_n$. Suppose $\Gamma_j$ is the true first candidate block for $[v]$. Since $\mathcal{P}_{[v]-\Gamma_{n}}(\beta) > \mathcal{P}_{[v]-\alpha}(\beta) = 0$ (assuming $k > 1$), we know that $j \geq n$. However, by Corollary \ref{cor_greedy}, we know that $\mathcal{Q}_{[v]-\Gamma_{n+1}}(\beta) < \mathcal{Q}_{[v]-\alpha}(\beta) = 0$. Thus $j \leq n$, so $j = n$. Since $\mathcal{P}_{[v]-\Gamma_{n}}(\beta) > 0$, the asymptotic greedy algorithm continues.

We may inductively apply the above argument, at the $\ell$th step replacing $[v]$ with $[v]-(\Theta_1+\dots+\Theta_{\ell-1})$ until $\ell = k$. This shows that $\alpha$ is the output of the asymptotic greedy algorithm.
\end{proof}

\begin{proof}[Proof of Theorem \ref{thm_converge}]
Suppose there exists an allowable representation $\alpha$ such that $\mathcal{P}_{[v]-\alpha}(\beta) = 0$. By Lemma \ref{lemma_greedy_true} we know that $\alpha$ is the output of the asymptotic greedy algorithm for $[v]$, that is $\alpha = \Omega_1 + \dots + \Omega_k$.

Suppose $\mathcal{P}_{[v]-\alpha}$ is a multiple of $f$. Then by Remark \ref{remark_derived_zero}, $\textnormal{Der}([v]-\alpha, H)$ is identically zero. Thus $\alpha$ must be the normalized gzd for all $D_s \in D$. Thus clearly $D$ converges to $\alpha$ in gzd. This handles Case 1.

Suppose $\mathcal{P}_{[v]-\alpha}$ is not a multiple of $f$. Let $K = \textnormal{Der}([v]-\alpha, H)$. Then $K$ is nonzero and satisfies the recurrence relation $f/(x-\beta)$ by Remark \ref{remark_coeff}. By Corollary \ref{cor_no_pos}, $f/(x-\beta)$ has no positive root, so by Theorem \ref{thm_oscillate}, $K$ is positive and negative infinitely often. Thus the sets $D_+$ and $D_-$ as defined in the statement of Theorem \ref{thm_converge} are infinite.

By Remark \ref{remark_asymptotic}, we know that for $s$ sufficiently large, the first $k-1$ blocks of the normalized gzd of $D_s$ are $\delta = \Omega_{1}+\dots+\Omega_{k-1}$. Suppose $\Omega_k = \Gamma_n$. 

First suppose $D_s \in D_+$. Since $K_s \geq 0$ and $\mathcal{P}_{[v]-(\Omega_1+\dots+\Omega_{k-1}+\Gamma_{n+1})}(\beta) < 0$, by Remark \ref{remark_asymptotic}, for $s$ sufficiently large, $\Gamma_{n}$ is the $k$th block of $\textnormal{NGZD}(D_s)$. Thus $\textnormal{NGZD}(D_s) = \alpha + \varepsilon_s$ for $s$ sufficiently large and $D_s \in D_+$. We claim that $\mathcal{L}(\varepsilon_s) \to -\infty$ as $s \to \infty$. Since $\mathcal{Q}_{[v]-\alpha-\Gamma_m}(\beta) < 0$ for any $m$, for $s$ sufficiently large $\Gamma_m$ is never the $(k+1)$th block selected by the greedy algorithm. Thus the claim is proved, so $D_+$ converges to $\alpha$ in gzd.

Now suppose $D_s \in D_-$. We first claim that $\Gamma_{n-1}$ is the $k$th block chosen by the greedy algorithm for all $s$ sufficiently large. This follows from the fact that $K_s < 0$ but $Q_{[v]-(\delta + \Gamma_{n-1})}(\beta) > 0$, and Remark \ref{remark_asymptotic}.

Let $\mathcal{R}': = \mathcal{R}(\delta+\Gamma_{n-1})$. We claim that the next block chosen is $\eta = S^{\mathcal{R}'-1}([c_1, \dots, c_t-1])$ for $s$ sufficiently large. By Proposition \ref{prop_greedy}, the largest possible block next chosen by the greedy algorithm is $\eta$ (since chosen blocks have non-intersecting supports). Thus if $\mathcal{Q}_{[v]-\delta-\Gamma_{n-1}-\eta}(\beta) > 0$, $\eta$ must always be chosen by the greedy algorithm for $s$ sufficiently large. Using Lemmas \ref{lemma_minus_adjacent} and \ref{lemma_repeating}, we have
\begin{equation*}
\begin{split}
\mathcal{Q}_{[v]-\delta-\Gamma_{n-1}-\eta}(\beta) &\ =\   \mathcal{Q}_{[v]-\delta-\Gamma_{n}}(\beta) + \mathcal{Q}_{\Gamma_{n}-\Gamma_{n-1}}(\beta) + \mathcal{Q}_{-\eta}(\beta) \\
    &\ >\  \beta^{\mathcal{R}(\Gamma_{n-1})} - \beta^{\mathcal{R}(\Gamma_{n-1})} \\
    &\ =\ 0.
\end{split}
\end{equation*}
At the $b$th step, the largest possible next block chosen is $\eta^b := S^{-bt}(\eta)$. However, by Lemma \ref{lemma_repeating}, $\mathcal{Q}_{[v]-\delta-\Gamma_{n-1}-\eta-\eta^1-\dots-\eta^b}(\beta) > 0$. This handles Case 2.

Lastly, Case 3 follows immediately from earlier discussion.
\end{proof}

\begin{remark}\label{remark_use_summand}
The utility of Theorem \ref{thm_converge} to the summand minimality question is as follows. Given $H_\sigma$, suppose we can find some string $v$ with non-negative entries such that we are not in Case 1 of Theorem \ref{thm_converge}. Then since $S^k([v])$ has a fixed number of summands, but in Cases 2 and 3 of Theorem \ref{thm_converge} the number of summands in the gzd of $S^k([v])$ grows arbitrarily large, $H_\sigma$ must not be summand minimal. As we shall see, in most cases, taking $v = (c_1+1)$ (where $(c_1, c_2, \dots, c_t)$ is the signature of $H$) suffices.
\end{remark}

%% file: summand_minimality.tex
\section{Summand minimality implies weakly decreasing signature} \label{section_summand_minimality}

Interestingly, if we have a PLRS whose signature does not satisfy the Parry condition and $c_1 \neq 1$, then we can find a simple example of a non-gzd representation with fewer summands than the gzd.

\begin{proposition} \label{prop_not_parry}
Suppose $H$ is a positive linear recurrence sequence whose signature does not satisfy the Parry condition and $c_1 \geq 2$. Then $H$ is not summand minimal.
\end{proposition}

\begin{proof}
Suppose the signature is $(c_1, \dots, c_t)$. Suppose $i$ is the smallest index such that 
$$(c_{i+1}, \dots, c_t, 0, \dots, 0)\ \geq_{\mathrm{lex}}\ (c_1, \dots, c_t)$$ 
(see Equation \eqref{eqn_parry} of Definition \ref{defn_parry}). Then there exists some $j \geq 1$ such that $c_{i+k} = c_{k}$ for $k < j$ and $c_{i+j} > c_j$. Consider the representation $(c_1, \dots, c_i+1, \underbrace{0, \dots, 0}_{j-1}, \infty_{t-1})$. This is clearly not an allowable sequence. By borrowing from the $j$th place (whose entry is $c_{i}+1$), we get 
$$(c_1, \dots, c_i, c_1, c_2, \dots, c_j, \infty_{t-1})\ =\ (c_1, \dots, c_i, c_{i+1}, c_{i+2}, \dots, c_{i+j-1}, c_j, \infty_{t-1}).$$ Since $c_j < c_{i+j}$, we get that the above representation is the gzd. The change in the number of summands is $c_1 + \dots + c_j - 1 \geq 1$ since $c_1 \geq 2$.
\end{proof}

\begin{example}
Let $\sigma = (5, 4, 2, 1, 5, 4, 3, 7)$. Then the representation $(5, 4, 2, 2, 0, 0, 0, \infty)$ has as its gzd $(5, 4, 2, 1, 5, 4, 2, \infty)$, which has more summands.
\end{example}

\begin{lemma}\label{lemma_parry}
Suppose $\sigma$ satisfies the Parry condition, the length of $\sigma$ is at least two, and $c_1 \geq 2$. Let $D = \textnormal{Der}((c_1+1), H)$. Then there exists an $N$ such that for $s > N$, we have $\textnormal{NGZD}(D_s) = S^1([1]) + \varepsilon_s$ with $\mathcal{L}(\varepsilon_s) < 0$.
\end{lemma}

\begin{proof}
We first prove the following inequalities and then explain how they imply the result:
\begin{eqnarray}
\beta & \ <\  &c_1 + 1, \label{eqn_beta_less} \\
2 \beta& \ >\ & c_1 + 1, \label{eqn_2_beta}\\
\beta & \ >\ & c_1. \label{eqn_beta_more}
\end{eqnarray}
Let $f$ be the minimal polynomial. Then
$$f(c_1) = c_1^{t} - c_1 c_1^{t-1} - c_2 c_1^{t-2} - \dots - c_t = - c_2 c_1^{t-2} - \dots - c_t < 0$$
since $\sigma$ has length at least two. Since $a \geq \beta \implies f(a) \geq 0$, we conclude that $c_1 < \beta$ (Equation \eqref{eqn_beta_more}).

Since $\sigma$ satisfies the Parry condition, $c_1 \geq c_t \geq 0$ for all $t$. Thus, if $|z| \leq c_1+1$,
\begin{equation*}
\begin{split}
|c_1 (z)^{t-1} + c_2 (z)^{t-1} + \cdots + c_t| &\ \leq\  c_1 (1 + (c_1+1) + \dots + (c_1+1)^{t-1}) \\
    & \ = \ (c_1+1)^t - 1.
\end{split}
\end{equation*}
Therefore we have $|z^t| > |c_1 z^{t-1} + c_2 z^{t-2} + \dots + c_t|$ for $|z| = c_1+1$; by Rouch{\'e}'s theorem, the polynomials $z^t$ and $f(z)$ have the same number of zeros in $|z| \leq c_1+1$, namely $t$ zeros. Thus $\beta < c_1 + 1$ (Equation \eqref{eqn_beta_less}).

Lastly, we have $\beta > c_1$ so $2 \beta > 2 c_1$ and $2 c_1 \geq c_1 + 1$ if $c_1 \geq 1$ (Equation \eqref{eqn_2_beta}).

By Equation \eqref{eqn_beta_less}, we get that $\mathcal{P}_{[c_1+1]-S^1([1])}(\beta) > 0$. Since $c_1 \geq 2$, we have that $(1)$ is an allowable block. Let $\Gamma_n = S^1([1])$. Note that $\Gamma_{n+1}$ must start with a $2$ (again since $c_1 \geq 2$). Then $\mathcal{P}_{[c_1+1]-\Gamma_{n+1}}(\beta) < \mathcal{P}_{[c_1+1]-S^1([2])}(\beta) < 0$ by Equation \eqref{eqn_2_beta}. Therefore $S^1([1])$ is the first candidate block for $[c_1+1]$ and the asymptotic greedy algorithm continues. Since $\mathcal{P}_{[c_1+1]-S^1([1])-[1]}(\beta) = -\beta + c_1 < 0$ by Equation \eqref{eqn_beta_more}, we know that the second candidate block must be less than $[1]$. In particular, the second block of $\textnormal{NGZD}(D_s)$ must be less than $[1]$ for $s$ sufficiently large, which is to say that $\mathcal{L}(\varepsilon_s) < 0$ as in the statement of the lemma.
\end{proof}

Note that Lemma \ref{lemma_parry} is not necessarily true if $c_1 = 1$. For example, if $\sigma = (1, 0, 0, 1)$, then $\textnormal{NGZD}(D_s) = S^2([1]) + \varepsilon_s$ for $s$ sufficiently large.

\begin{theorem} \label{thm_parry}
Suppose $\sigma$ satisfies the Parry condition and $c_1 \neq 1$ and $\sigma$ is not weakly decreasing. Then $H = H_\sigma$ is not summand minimal.
\end{theorem}

\begin{proof}
Let $D = \textnormal{Der}(c_1+1, H)$. We shall show that there exists an $s$ such that the number of summands in $D_s$ is greater than $c_1 + 1$. By Remark \ref{remark_use_summand}, if we are in Case 2 or 3 of Theorem \ref{thm_converge}, then we are done. Suppose instead that we are in Case 1 of Theorem \ref{thm_converge}.  We proceed by contradiction.

If we are in Case 1, then there exists an $\alpha$ such that $g := -\mathcal{P}_{[v]-\alpha}$ is a multiple of $f$, the minimal polynomial of $H$. By Lemma \ref{lemma_parry}, 
$$g\ = \ x^{r+1} - (c_1+1) x^{r} + b_1 x^{r-1} + \dots + b_{r}$$
for some $b_i \geq 0$. The number of summands in $\alpha$ is $\mathcal{S} := 1 + \sum b_i$. If $\mathcal{S} > c_1 + 1$, then $H_\sigma$ is not summand minimal, so we may assume $\mathcal{S} \leq c_1 + 1$.

We know that $g = f h$ for some $h$. Since the leading coefficient of $f$ is 1, by long division of polynomials we conclude that $h$ must have integer coefficients (this is essentially Gauss' lemma). Let $\Sigma (p)$ represent the sum of coefficients of a polynomial $p$. Notice that $\Sigma (f h) = \Sigma (f) \Sigma (h)$. Note that $\Sigma(f h) = \Sigma(g) = \mathcal{S}-(c_1+1) \leq 0$. We know $\mathcal{S} \geq 1$, with equality if and only if $b_i = 0$ for all $i$. However, if $b_i = 0$ for all $i$, then $g = x - c_1$ and $\sigma = (c_1)$, contradicting the fact that $\sigma$ is not weakly decreasing. Thus $\mathcal{S} \geq 2$.

Since $\mathcal{S} \geq 2$, we have $|\Sigma(g)| \leq c_1 - 1$.  We know that $|\Sigma(f)| = |1-c_1-\dots-c_t| \geq c_1$ since $t \geq 2$ and $c_t \geq 1$. Since $|\Sigma(g)|$ is an integer multiple of $|\Sigma(f)|$, but $|\Sigma(g)| \leq c_1 - 1 < c_1 \leq |\Sigma(f)|$, we must have that $\Sigma(g) = 0$, which is to say that $\mathcal{S} = c_1+1$ and $x = 1$ is a root of $g$.

When we factor out $x-1$ from $g$, we are left with
\begin{equation*}
g = (x-1)(x^{r} - c_1 x^{r-1} - (c_1 - b_1) x^{r-2} - (c_1 - b_1 - b_2) x^{r-3} - \dots - (c_1 - b_1 - \dots - b_{r-1})).
\end{equation*}
Let $q = g/(x-1)$. Notice that $c_1 \geq c_1 - b_1 \geq \dots \geq c_1 - b_1 - \dots - b_{r-1} > 0$. Therefore by a theorem of Brauer (Theorem 2 of \cite{Br}), $q$ is irreducible in $\mathbb{Q}[x]$. Since $g$ has a factorization over $\mathbb{Q}[x]$ as $g = f h$, we must have that $f = q$ and $h = (x-1)$. However, this contradicts the fact that the signature was assumed to not be weakly decreasing. Thus either we are in Case 1 of Theorem \ref{thm_converge} but $H_\sigma$ is not summand minimal, or else we are not in Case 1 of Theorem \ref{thm_converge}, in which case $H_\sigma$ is again not summand minimal from previous discussion.
\end{proof}

Lastly, we need to check the case when $c_1 = 1$. We again examine the gzd of elements in $D = \textnormal{Der}(2, H)$ (note that $c_1 + 1 = 2$). For a given $H = H_\sigma$, if we are in Cases 2 or 3 of Theorem \ref{thm_converge}, then $H$ is not summand minimal. Thus we assume we are in Case 1. Suppose there exists an $\alpha$ such that $g = \mathcal{P}_{[2]-\alpha}$ is a multiple of $f$. If $H$ is summand minimal, then $0 \leq \Sigma(g) \leq 2$. Clearly $\Sigma(g) \neq 2$ (since then the gzd of $D_s$ would have no nonzero terms for all $s$). Suppose $\Sigma(g) = 1$. Then $\alpha = S^r([1])$ for some $r$. We cannot have $r \leq 0$ since the terms of $H$ are monotonically increasing. Thus in such a case $g = x^r - 2$. Then $g$ has all roots of the same modulus, so $g$ cannot be a multiple of $f$ (which has one root of strictly greater modulus than all the rest), unless $r = 1$, in which case the signature is $(2)$ contradicting the fact that $c_1 = 1$. Thus $\Sigma(g) = 0$. We clearly can't have $\alpha = S^r([2])$ since $[2]$ is not an allowable block. Thus, $\alpha = S^a([1]) + S^b([1])$ for some $a$ and $b$, so $2 H_n = H_{n+a} + H_{n+b}$. We cannot have one of $a$ or $b$ (or both) equal to 0 since then $H_n = H_{n+a}$. We can't have both $a < 0$ and $b < 0$ because then $H_n > H_{n+a}$ and $H_n > H_{n+b}$, so $2 H_n > H_{n+a} + H_{n+b}$. Similarly, we cannot have both $a > 0$ and $b > 0$. Thus one is greater than zero and one is less. Thus $g = x^r - 2 x^s + 1$ with $r > s$. We utilize the following theorem of Schinzel \cite{Schi} which appears in English in \cite{FS2}.

\begin{theorem}[Theorem 1 of \cite{Schi}] \label{thm_irr}
Let $g(r,s)=x^r-2x^s+1$ with $r, s \in \mathbb{N}$ and $r>s$. Let $d = \gcd(r, s)$. The polynomial
\begin{gather*}
h(r, s)\ =\ \frac{g(r, s)}{x^d - 1} \ = \ x^{r-d} + x^{r-2d} + \dots + x^{r} - x^{r-d} - x^{r-2d} - \dots - 1
\end{gather*}
is irreducible for all $r$, $s$ except for $(r, s) = (7k, 2k)$ or $(7k, 5k)$, in which case $h(r, s)$ factors into irreducible pieces
\begin{gather*}
h(r, s)\ =\ (x^{3k} + x^{2k}-1)(x^{3k}+x^k+1) \textnormal{ \ \ and \ \ } (x^{3k}+x^{2k}+1)(x^{3k}-x^k-1),
\end{gather*}
respectively.
\end{theorem}

\begin{proposition}\label{prop_big}
Suppose $f$ is of positive type. Then $f \nmid g(r, s)$ for any $r$, $s$ except for $r = s+1$, in which case $f = x^{r-1}-x^{r-2} - \dots - 1$.
\end{proposition}

\begin{proof}
For now assume $(r, s) \neq (7k, 2k)$ of $(7k, 5k)$. Suppose $f \ | \ g(r, s)$. We know that $f$ cannot be a product of cyclotomic polynomials since $\beta > 1$ by Proposition \ref{prop_pos_type}. Therefore, $h(r, s) \ | \ f$. 

If $d = 1$, then $g(r, s) = h(r, s) (x-1)$. We know that $x-1 \nmid f$ because $f$ has only one positive root, $\beta$. Thus if $f \ |  \ g(r, s)$, then $f = h(r, s)$ which is impossible unless $r = s+1$.

Suppose $d > 1$. Suppose $f = h(r, s) \cdot p$ with $p \ | \ x^d-1$ and $x-1 \nmid p$. Then $p$ must have symmetric coefficients (products of cyclotomic polynomials other than $x-1$ have symmetric coefficients). Suppose $p = x^\ell + a_1 x^{\ell-1} + \cdots + a_1 x + 1$. Then
$$ f(x)\ =\ (x^{r-d} + x^{r-2d} + \cdots + x^r - x^{r-d} - \cdots - 1)(x^\ell + a_1 x^{\ell-1} + \cdots + a_1 x + 1).$$
The coefficient of $x^{r-d+\ell-1}$ in $f$ is $a_1$ since the $x^{r-d-1}$ coefficient of $h(r, s)$ is 0, and the coefficient of $x$ in $f$ is $-a_1$ since the $x$ coefficient of $h(r, s)$ is 0. Since $f$ is of positive type and of degree $x^{r-d}$, the $x^{r-d-1}$ coefficient must be negative implying that $a_1 < 0$. Then $-a_1 > 0$, implying that the $x$ coefficient of $f$ is positive, contradicting the fact that $f$ is of positive type. Thus in fact $f \nmid g(r, s)$.

Lastly, we must handle the case when $(r, s) = (7k, 2k)$ or $(7k, 5k)$. Suppose $(r, s) = (7k, 2k)$. Let $I_1 = x^{3k} + x^{2k} - 1$ and $I_2 = x^{3k}+x^k+1$. The previous argument still implies that $f \neq h(r, s) \cdot p = I_1 \cdot I_2 \cdot p$. Thus, if $f \ | \ g(r, s)$, then either $f = I_1 \cdot p$ or $f = I_2 \cdot p$. Suppose $f = I_1 \cdot p$. Notice that $I_1 = g(4k, k)/(x^k -1) = h(4k, k)$, so the previous argument again applies. Suppose instead that $f = I_2 \cdot p$. Neither $I_2$ nor $p$ has positive roots, so $f \neq I_2 \cdot p$. The argument for $(r, s) = (7k, 2k)$ is virtually identical.
\end{proof}

\begin{corollary} \label{cor_c_1}
Suppose $\sigma$ has $c_1 = 1$ and is not weakly decreasing. Then $H = H_\sigma$ is not summand minimal.
\end{corollary}

\begin{proof}
Let $f$ be the minimal polynomial. Let $v = (2)$. By the previous discussion, if we are in Case 1 of Theorem \ref{thm_converge} and $H$ is summand minimal, then there exists some polynomial $g(r, s)$ which is a multiple of $f$. However, by Proposition \ref{prop_big}, this is impossible. Therefore either we are in Case 1 but $H$ is not summand minimal, or else we are in Cases 2 or 3 of Theorem \ref{thm_converge}, so the number of summands in the gzd of elements of $\textnormal{Der}(2, H)$ grows arbitrarily large, and in particular larger than 2. Therefore in all cases $H$ is not summand minimal.
\end{proof}

Combining Proposition \ref{prop_not_parry}, Theorem \ref{thm_parry}, and Corollary \ref{cor_c_1}, we conclude that if $\sigma$ is not weakly decreasing, then $H_\sigma$ is not summand minimal.

%% file: conclusion.tex
\section{Concluding remarks} \label{section_conclusion}
Suppose $\beta$ is the dominating root of the polynomial of positive type $x^t - c_1 x^{t-1} - \dots - c_1$ and let $\sigma = (c_1, \dots, c_t)$. If $\sigma$ satisfies the Parry condition, we say that $\beta$ is a \textit{simple Parry number} (in the $\beta$-expansion literature, $\sigma$ is sometimes denoted by $d(1, \beta)$). Extrapolating from our previous definition, we say that $\beta$ is \textit{summand minimal} if for all $x \in \mathbb{Z}_+[\beta^{-1}]$, the number of summands in the $\beta$-expansion of $x$ is less than or equal to the number of summands in any other representation of $x$ as a sum of positive integral multiples of powers of $\beta$. We note that the techniques of this paper immediately demonstrate the following.
\begin{proposition}
Suppose $\beta$ is a simple Parry number whose corresponding polynomial is irreducible in $\mathbb{Q}[x]$. Then $\beta$ is summand minimal if and only if $\sigma$ is weakly decreasing.
\end{proposition}

One may then ask
\begin{question}
Are there simple Parry numbers $\beta$ so that $\beta$ is summand minimal but $d(1, \beta)$ is not weakly decreasing?
\end{question}
\noindent We conjecture that the answer to this question is no.

It is clear that summand minimality implies that
\begin{equation*} \label{eqn_property_f_prime}
\textnormal{Fin}(\beta)\ =\ \mathbb{Z}_+[\beta^{-1}].
\end{equation*}
One may also wonder about the related issue of whether
\begin{equation} \label{eqn_property_f}
\textnormal{Fin}(\beta)\ =\ \mathbb{Z}[\beta^{-1}]_+.
\end{equation}
If $\beta$ is a simple Parry number, then $\mathbb{Z}_+[\beta^{-1}] = \mathbb{Z}[\beta^{-1}]_+$ (see \cite{FS}). Equation \eqref{eqn_property_f} is known in the literature as property (F). Frougny and Solomyak \cite{FS} were the first to systematically investigate which $\beta$ satisfy property (F), with Theorem \ref{thm_fs} being the first definitive result on the matter. Notice that a simple Parry number $\beta$ with $d(1, \beta) = (c_1, \dots, c_t)$ satisfies property (F) if and only if for all $v$ such that $\mathcal{P}_{v}(\beta) > 0$, we end up in Cases 1 or 2 of Theorem \ref{thm_converge} (with $H = H_{(c_1, \dots, c_t})$). To the best of the authors' knowledge, to date there does not exist a simple characterization of those $\beta$ which satisfy property (F) (see, e.g., \cite{ABBPT}). Thus it is interesting that, at least in the gzd case, the issue of summand minimality has a simple characterization (weakly decreasing signature), but the related question of property (F) for $\beta$-expansions has an apparently more complicated answer.

Though this paper completely resolves the issue of summand minimality of PLRSs, there are several finer points to still be addressed. In Proposition \ref{prop_not_parry}, we exhibit just one number for which there exists a non-gzd representation with fewer summands than the gzd. However, in Theorem \ref{thm_parry} and Corollary \ref{cor_c_1}, we show that $(c_1+1)H_n$ has more summands than the gzd infinitely often. One may then ask
\begin{question}
If $\sigma$ does not satisfy the Parry condition, does $(c_1+1)H_n$ have more summands than the gzd infinitely often?
\end{question}

\noindent We conjecture that this question has a positive answer. One may also ask

\begin{question}
Given any $\sigma$ which is not weakly decreasing, is there a simple characterization of those $m$ for which $H_\sigma$ is not summand minimal?
\end{question}

Theorem \ref{thm_converge} itself brings up many interesting further points of investigation. A natural question is to quantify the rate at which the derived sequence converges in gzd. We personally find Case 2 of Theorem \ref{thm_converge} to be particularly interesting. In fact, we know of very few examples of pairs $v$ and $H_\sigma$ which are in Case 2. Clearly a necessary condition is that $\mathcal{P}_{(1, -c_1, \dots, -c_t)}$ be reducible in $\mathbb{Q}[x]$. It would be interesting to find nontrivial families of examples of Case 2.

As noted in the introduction, Theorem \ref{thm_converge} is closely related to Theorem \ref{thm_grabner}. However, Theorem \ref{thm_grabner} also asserts that $\textnormal{Der}((k), H_\sigma)$ converges in gzd to the $\beta$-expansion of $k$. We note more generally that if $\sigma$ satisfies the Parry condition, then for any $v$, we have that $\textnormal{Der}(v, H_\sigma)$ converges in gzd to the $\beta$-expansion of $Q_{[v]}(\beta)$ (when in Case 2, $\alpha$ as in the statement of Theorem \ref{thm_converge} is the $\beta$-expansion). From a theorem of Parry \cite{Pa}, the representation that $\textnormal{Der}(v, H_\sigma)$ converges to is the $\beta$-expansion of some number, and it is straightforward to see that it represents $Q_{[v]}(\beta)$. When $\sigma$ does not satisfy the Parry condition, then for any $v$ we have that $\textnormal{Der}(v, H_\sigma)$ still converges to \textit{a} $\beta$-expansion of $Q_{[v]}(\beta)$, though not \textit{the} $\beta$-expansion. Thus Theorem \ref{thm_converge} implies that every element in $\mathbb{Z}[\beta^{-1}]_+$ has a unique $\beta$-representation $\sum r_i \beta^i$ where the representation $(\dots, r_i, \dots)$ is composed of allowable blocks (and this representation differs from the usual $\beta$-expansion when the signature does not satisfy the Parry condition). Future investigations may involve exploring these and other implications of Theorem \ref{thm_converge} regarding the connection between $\beta$-expansions and generalized Zeckendorf decompositions.

%% file: appendix.tex
\appendix
\section{Algorithm from any representation to the gzd}
\label{section_algorithm}
\begin{definition}\label{def:legal}
Let $\rho = (r_n, \dots, r_0, \infty_{t-1})$ be a representation using $H_\sigma$. We say $\rho$ is \textit{legal up to $\mathbf{s}$} if $(r_n, \dots, r_s)$ can be expressed as $\Lambda_1 \oplus \dots \oplus \Lambda_j$ with each $\Lambda_i$ an allowable block and $\Lambda_1 \neq (0)$ (where $\oplus$ represents concatenation).
\end{definition}

\begin{definition}\label{def:mli}
The \textit{minimum legal index (m.l.i.)} of $\rho$ is the smallest index $s$ such that $\rho$ is legal up to $s$.
\end{definition}

Notice that if $\sigma = (c_1, \dots, c_t)$ and $\rho$ is a representation whose m.l.i.\ is $s$, then $r_{s-1} \geq c_1$. If $r_{s-1} = c_1$, then $r_{s-2} \geq c_2$. At some point, we must either have that $r_{s-j} > c_j$, or for all $1 \leq j < t$, we have $r_{s-j} = c_j$ and $r_{s-t} \geq c_t$. This motivates the following definition.

\begin{definition}\label{def:vio}
Suppose $\rho$ has m.l.i.\ equal to $s$. Let $j$ be the smallest index such that $r_{s-j} >c_j$, or if $r_{s-i}=c_i$ for all $1 \leq i \leq t$, then let $j=t$. The \textit{violation index} is $s-j$ and we call $r_{s-j}$ the \textit{violation}.
\end{definition}

\begin{definition}\label{def:prefix}
Suppose $\rho$ has m.l.i.\ equal to $s$ and violation index equal to $j$. Then $(r_{s-1}, \dots, $ $r_{s-(j-1)}) = (c_1,\dots,c_{j-1})$ is called the \textit{violation prefix}. The prefix and the violation together comprise the \textit{violation block}.
\end{definition}

\begin{remark}\label{rem_left_neighbor_block}
A representation $\rho$ can be decomposed as
\begin{equation}\label{eqn_left_neighbor}
\rho\ = \ \Lambda_1 \oplus \cdots \oplus \Lambda_j \oplus \Psi \oplus (r_d, \dots, r_0, \infty_{t-1}),
\end{equation}
with each $\Lambda_i$ a valid block, and $\Psi$ the violation block. For any $\Lambda_i$ (or $\Psi$), we call the block immediately to the left of $\Lambda_i$ (or $\Psi$) as in Equation \eqref{eqn_left_neighbor} the \textit{left neighbor block} of $\Lambda_i$ (or $\Psi$). The left neighbor block of $\Lambda_1$ is defined to  be $(0)$.
\end{remark}

\begin{definition}\label{def:semleg}
Suppose $\rho$ has m.l.i.\ $s$ and violation index $j$. We say that $\rho$ is \textit{semi-legal} up to $q$ if $q = s - j + 1$. We call $q$ the \textit{semi-legal index (s.l.i.)}.
\end{definition}

\begin{remark}
We note that the s.l.i.\ is the index to the left of the violation index, i.e., the s.l.i.\ is the violation index plus one. Furthermore, the difference between the m.l.i.\ and the s.l.i.\ is exactly the length of the violation prefix, which is $0$ if the violation prefix is empty.
\end{remark}

\begin{remark}\label{rem_close_gzd}
The m.l.i.\ and the s.l.i.\ give us a way to quantify how ``close'' a representation of $m$ is to $\textnormal{GZD}(m)$. Definitions \ref{def:legal}, \ref{def:mli} and \ref{def:semleg} imply that the m.l.i.\ and the s.l.i.\ of $\rho$ are both equal to $0$ if and only if $\rho$ is allowable.
\end{remark}

\begin{example}\label{ex:abletocarry}
Let $\sigma=(3,2,4)$ and $\rho = (3, 2, 1, 1, 3, 0, 3, 3, 5, \infty_2)$. Then
$$ \rho\ =\  \left( \, \boxed{3, 2, 1}, \boxed{1}, \boxed{3, 0}, \lboxed{3,3}, 5, \infty_2 \, \right),$$
where each closed box represents a valid block, and the right-opened box ($\lboxed{3,3}$) represents the violation block. The violation index is $0$, the s.l.i.\ is $1$, and the m.l.i\ is $3$. We are able to carry to $3$ (the m.l.i.) because $r_2=3=c_1$, $r_1=3>2=c_2$ and $r_0=5>4=c_3$. After carrying, we get
$$\Big( \, \boxed{3, 2, 1}, \boxed{1}, \boxed{3, 1}, \boxed{0},\boxed{1}, \boxed{1}, \infty_2 \,\Big),$$
which is the gzd. We note that the m.l.i.\ and s.l.i.\ equal 0.
\end{example}

Now consider the same signature as Example \ref{ex:abletocarry} but with $\rho=(3,2,1,1,3,0,3,3,1, \infty_2)$. The violation index, m.l.i., and s.l.i.\ are still the same but we are not able to carry because $r_0=1<4=c_3$. This motivates the following definitions.

\begin{definition}\label{def:coi}
Suppose $\rho$ has m.l.i.\ equal to $s$. We call $s-\ell$ the \textit{carry obstruction index (c.o.i.)} if for all $1 \leq i < \ell$, we have $r_{s-i} \geq c_i$ and $r_{s-\ell} < c_\ell$.
\end{definition}

\begin{definition}\label{def:rei}
Suppose $\rho$ has m.l.i.\ equal to $s$. Then $s-e$ is called the \textit{rightmost excess index (r.e.i.)} if $e$ is the largest index such that for all $i < e$, we have $r_{s-i} \geq c_i$ and $r_{s-e} > c_e$.
\end{definition}

\begin{example}\label{ex:make_coi_large}
Consider the aforementioned example with $\sigma=(3,2,4)$ and $\rho = (3, 2, 1, 1, 3,$ $ 0, 3, 3, 1, \infty_2)$. Here, m.l.i.\ $=3$, s.l.i.\ $=2$ and the violation index is $1$. The c.o.i.\ is $0$ and the r.e.i.\ is $1$ because $r_1=3>c_2=2$. We can borrow from the r.e.i.\ to make our c.o.i.\ large enough that we are able to carry. We demonstrate this in Table \ref{table:coi}.

\begin{table}[H]
\centering
\begin{tabular}{|c||c|c|c|c|c|c|c|c|c|c|c|}
\hline
Remark ($\downarrow$)/Index ($\to$) & ${8}$ & ${7}$ &   ${6}$ &   ${5}$ &   ${4}$ &   ${3}$ &   ${2}$    &   ${1}$ &   ${0}$ &  ${-1}$  &   ${-2}$ \\ \hline \hline
${\rho}$ & $3$ &   $2$ &   $1$ &   $1$ &   $3$ &   $0$    &   $3$ &   $3$ &   $1$ &   $\infty$    &   $\infty$ \\
Borrow from 1.	&	~ &   ~ &   ~ &   ~ &   ~ &   ~    &   ~ &   $-1$ &   $3$ &   $2$    &   $4$ \\ \hline
~	&	$3$ &   $2$ &   $1$ &   $1$ &   $3$ &   $0$    &   $3$ &   $2$ &   $4$ &   $\infty$    &   $\infty$ \\
  Carry to 3.	&	~ &   ~ &   ~ &   ~ &   ~    &  $1$ &   $-3$ &   $-2$ &   $-4$    &   ~ & ~ \\ \hline
~	&	$3$ &   $2$ &   $1$ &   $1$ &   $3$ &   $1$    &   $0$ &   $0$ &   $0$ &   $\infty$    &   $\infty$ \\ \hline
\end{tabular}
\caption{Sequence of borrows and carries to move to the gzd.}
\label{table:coi}
\end{table}
\end{example}



Remark \ref{rem_close_gzd} explains how a representation can be thought of as being ``far'' from the gzd if its m.l.i.\ and s.l.i.\ are large. As such, a potential way to turn any representation into the gzd is to try decreasing the m.l.i.\ and s.l.i.\ of the representation to zero. To do so, one may start by finding the first violation and attempting to ``fix'' it. Because any valid block is lexicographically less than $\sigma$ (see Definition \ref{def:block}), the entry at the violation index is always ``too large.'' This suggests that we may try to carry in order to fix it, as in Example \ref{ex:abletocarry}. In the case where we are not able to carry, which is to say that the c.o.i.\ exists, we would first borrow from the r.e.i.\ in order to carry, as in Example \ref{ex:make_coi_large}. As one performs these borrows and carries to fix all possible violations, one would expect to decrease the m.l.i.\ and the s.l.i.\ to zero to reach the gzd. Motivated by this intuition, we introduce Algorithm \ref{alg_gzd}.  Although there is some subtlety in the algorithm, this intuition is the key idea in the proof of validity.


\begin{alg}\label{alg_gzd} \

\begin{algorithm}[H]
 \textbf{Input}:{ $\sigma$ (the signature) and $\rho$ (a representation of $m$ using $H_\sigma$)} \\
 \textbf{Output}:{ the gzd of $m$} \\
 \While{the m.l.i.\ is not 0}{
  \eIf{able to carry to m.l.i.}{
   carry to m.l.i.\\
   \While{left neighbor block is $(c_1, \dots, c_t)$}{
   	carry to next block
   }
   }{
   borrow from the r.e.i.
  }
 }
\end{algorithm}
\end{alg}

\begin{example}
Let $\sigma = (3, 2, 5)$. We apply Algorithm \ref{alg_gzd} to $\rho = (2, 3, 2, 4, 4, 3, 0, 0, \infty_2)$.

\begin{center}
\begin{tabular}{| l || l | l | l | l | l | l | l | l | l | l |}
\hline Remark ($\downarrow$)/Index ($\to$)                                                                                & 7 & 6  & 5  & 4  & 3  & 2  & 1  & 0  & -1       & -2      \\
\hline \hline
m.l.i.\ $=4$, s.l.i.\ $=4$, c.o.i.\ $=1$, r.e.i.\ $=2$. & 2 & 3  & 2  & 4  & 4  & 3  & 0  & 0  & $\infty$ & $\infty$ \\

Borrow from $2$.                                                                                                   &   &    &    &    &    & -1 & 3  & 2  & 5        &          \\
\hline
m.l.i.\ $=4$, s.l.i.\ $=4$, c.o.i.\ $=1$, r.e.i.\ $=3$. & 2 & 3  & 2  & 4  & 4  & 2  & 3  & 2  & $\infty$ & $\infty$ \\
Borrow from $3$.                                                                                                   &   &    &    &    & -1 & 3  & 2  & 5  &          &          \\
\hline
Able to carry.                                                                                                     & 2 & 3  & 2  & 4  & 3  & 5  & 5  & 7  & $\infty$ & $\infty$ \\
Carry to $4$.                                                                                                      &   &    &    & 1  & -3 & -2 & -5 &    &          &          \\
\hline
Left neighbor block $= \sigma$.                                                                                    & 2 & 3  & 2  & 5  & 0  & 3  & 0  & 7  & $\infty$ & $\infty$ \\
Carry to $7$.                                                                                                      & 1 & -3 & -2 & -5 &    &    &    &    &          &          \\
\hline
m.l.i.\ $=1$, s.l.i.\ $=1$, able to carry.                                           & 3 & 0  & 0  & 0  & 0  & 3  & 0  & 7  & $\infty$ & $\infty$ \\
Carry to $1$.                                                                                                      &   &    &    &    &    &    & 1  & -3 & -2       & -5       \\
\hline
m.l.i.\ $=1$, s.l.i.\ $=1$, able to carry.                                           & 3 & 0  & 0  & 0  & 0  & 3  & 1  & 4  & $\infty$ & $\infty$ \\
Carry to $1$.                                                                                                      &   &    &    &    &    &    & 1  & -3 & -2       & -5       \\
\hline
m.l.i.\ $=0$. We've reached the gzd.                                                                                  & 3 & 0  & 0  & 0  & 0  & 3  & 2  & 1  & $\infty$ & $\infty$\\
\hline
\end{tabular}
\end{center}
\end{example}

To assist in the proof that Algorithm \ref{alg_gzd} terminates in the gzd, we state and prove a few lemmas. We first show that the s.l.i.\ weakly decreases in Lemma \ref{lem_sli_noinc}. Then, we show that the s.l.i.\ strictly decreases after finitely many steps in Lemma \ref{lem_sli_dec}. Finally, using those two lemmas, we prove that the m.l.i.\ decreases to $0$ in finitely many steps.

\begin{lemma}\label{lem_sli_noinc}
The semi-legal index (s.l.i.) monotonically decreases during Algorithm \ref{alg_gzd}.
\end{lemma}

\begin{proof}
When performing Algorithm \ref{alg_gzd}, if we are not able to carry to the m.l.i.,\ then the s.l.i.\ either stays the same or decreases. Thus, for the purposes of proving the lemma, we may suppose that we are able to carry.

Suppose the left neighbor block of the violation block is $(c_1, \dots, c_{\ell-1}, d_\ell)$ with $d_\ell < c_\ell$. If $d_\ell < c_\ell -1$, then we still have a valid block after carrying. The entries of $\rho$ with indices strictly between the old m.l.i.\ and the old s.l.i.\ become zero after carrying, and thus become valid blocks. Therefore the s.l.i.\ will have either stayed the same or decreased.

Now instead suppose that $d_\ell = c_\ell - 1$ and $\ell < t$. After carrying, our representation is of the form
$$\Lambda_1 \oplus \dots \oplus \Lambda_j \oplus (c_1, \dots, c_{\ell-1}, c_{\ell}, 0, \dots, 0, v-c_{k+1}, \dots),$$
where $v$ is the original violation and the length of $0,\dots, 0$ in the middle is the length of the violation prefix before carrying.

We cannot have a violation before $v-c_{k+1}$ because a violation requires that an entry is greater than its corresponding entry in the signature; however, the first $\ell$ terms agree with the signature and the remaining terms are all zero, so they are either equal to or less than the corresponding terms in the signature. Therefore the earliest possible violation is at $v-c_{k+1}$, so in this case the s.l.i.\ either stays the same or decreases.

Finally, suppose $\ell = t$ and $d_\ell = c_t - 1$. After we carry, the left neighbor block of the violation block is now $(c_1, \dots, c_t)$, so we immediately carry again by Algorithm \ref{alg_gzd}. After the carry, our representation looks like
$$\Lambda_1 \oplus \dots \oplus \Lambda_{j-1} \oplus (c_1, c_2, \dots, c_m - r, \underbrace{0, \dots, 0}_{t}, \dots),$$
with $r \geq 0$. If $r \geq 1$, then clearly the s.l.i.\ has not increased. If $m < t$ and $r = 0$, then since we have $t$ trailing zeros, the s.l.i.\ still has not increased ($(c_1, \dots, c_m, \underbrace{0, \dots, 0}_{t})$ is composed of allowable blocks). If $r = 0$ and $m = t$, then by Algorithm \ref{alg_gzd}, we must carry yet again to the next left neighbor block and repeat the above arguments. We know that at some point the left neighbor block is not $(c_1, \dots, c_t - 1)$ since there are finitely many nonzero blocks. As such, at some point this process terminates without increasing the s.l.i.\
\end{proof}

\begin{lemma}\label{lem_sli_dec}
Suppose  s.l.i.\ $\geq 1$. Then, after finitely many steps of Algorithm \ref{alg_gzd}, the s.l.i.\ decreases.
\end{lemma}

\begin{proof}
Let $\rho$ be
$$\rho\ =\ \Lambda_1 \oplus \dots \oplus \Lambda_j \oplus (c_1, \dots, c_m, v, \dots), $$
where $0\leq m<t$, the violation is $v$, and each $\Lambda_i$ is a valid block. Suppose $\Lambda_j=(c_1,\dots,c_{\ell-1},d_\ell)$, with $d_\ell<c_\ell$.

We now show that after finitely many steps, the value of $v$ decreases. In performing Algorithm \ref{alg_gzd}, every time we borrow, the value in the rightmost excess index (r.e.i.) decreases by 1. If we keep borrowing and are never able to carry, then at some point we must have decreased the value at the original r.e.i.\ to the point where it is no longer ``excess''. At this point, the r.e.i.\ increases, and so if we never carry, then after finite time the r.e.i.\ becomes equal to the violation index. In that case, the next time we borrow, $v$ decreases.

We are left to handle the case when we are able to carry. When we carry, $v$ decreases by $c_{m+1} \geq 0$. If this amount is $c_{m+1} \geq 1$, then clearly $v$ decreases. Otherwise, $c_{m+1} = 0$ (implying $m \geq 1$). After carrying, we obtain
$$ \Lambda_1 \oplus \dots \oplus \Lambda_{j-1} \oplus (c_1, \dots, c_{\ell-1}, d_{\ell}+1, \underbrace{0,\dots,0}_{m},v,\dots).$$
There are now $m \geq 1$ zeros to the left of $v$. By Lemma \ref{lem_sli_noinc}, after performing any possible carries in the left neighbor blocks, the s.l.i.\ will at worst stay the same, in which case either the s.l.i.\ decreases (as we ultimately want) or $v$ remains the violation. 

Assuming $v$ remains the violation, then repeating the above arguments, we either have that at some point we borrow from $v$, or else we carry. In the former case, $v$ obviously decreases; in the latter case, either $v$ decreases or the violation block is of the form $(c_1, \dots, c_q, v)$ with $c_{q+1} = 0$. However at this point, since we've already carried once before (resulting in $m$ zeros to the left of $v$), we have that $c_{q-m+1} = \dots = c_q = 0$. Since $c_1 \neq 0$, we must have that $q - m \geq 1$. Then, after the second carry, the number of zeros to the left of $v$ increases by $q-m \geq 1$.

As Algorithm \ref{alg_gzd} continues, we may repeat the above arguments implying that either the s.l.i.\ decreases, or $v$ eventually decreases, or the number of zeros preceding $v$ grows larger and larger. In the last case, eventually the number of zeros preceding $v$ grows larger than $t$. Then at the next step, either the s.l.i.\ decreases, or the violation prefix has zero length, in which case $v$ decreases by $c_1 \geq 1$ after the next carry.

Thus in all cases, either the s.l.i.\ decreases, or $v$ decreases. However, if $v$ continues to decrease, then after finitely many steps $v$ must be small enough that the s.l.i.\ decreases.
\end{proof}

\begin{theorem}\label{thm_alg_term}
Algorithm \ref{alg_gzd} terminates in the gzd.
\end{theorem}

\begin{proof}
By Lemma \ref{lem_sli_dec}, the s.l.i.\ decreases to zero. Theorem \ref{thm_alg_term} holds if and only if the m.l.i.\ decreases to zero after finitely many steps. Therefore, we need only show that when the s.l.i.\ is zero, the m.l.i.\ goes to zero.

Suppose the s.l.i.\ is zero. If the m.l.i.\ is not equal to zero, then our representation is of the form $\rho=\Lambda_1 \oplus \dots \oplus \Lambda_j \oplus (c_1, \dots, c_m, \infty_{t-1})$. Suppose $\Lambda_j = (c_1, \dots, c_{\ell-1}, d_\ell)$  with $d_\l<c_\l$. We can immediately carry to the m.l.i.\ using the $\infty$ places, resulting in
$$ \Lambda_1 \oplus \dots \oplus \Lambda_{j-1} \oplus(c_1, \dots, c_{\ell-1}, d_\ell+1, \underbrace{0,\dots,0}_{m},\infty_{t-1}).$$
If $d_\l+1<c_\l$, then the m.l.i.\ is zero.

Otherwise, we have $d_\l+1=c_\l$. If $\ell = t$, then we carry to the left neighbor block as needed which ultimately decreases the m.l.i.\ to zero. If $\ell < t$, then unless $c_{\ell+1} = \dots = c_{\ell+m} = 0$, we are done. Otherwise, we must carry again, and the violation prefix grows larger. Repeating the above arguments, we see that either the algorithm terminates with the m.l.i.\ equal to zero, or else the violation prefix grows arbitrarily large. However, the violation prefix cannot grow larger than $t$, so we must eventually end up in the former case.
\end{proof}

\section{Proof of Proposition \ref{prop_min_poly}}\label{appendix_min_poly}

\begin{definition}
Let $H$ be a sequence. Define $H_{n, k}$ by
\begin{equation*}
H_{n, k}\ =\ 		\begin{bmatrix}
			H_n	&		H_{n+1}	&	\cdots	&		H_{n+k} \\
			H_{n+1}	&	H_{n+2}	&	\cdots	&		H_{n+k+1}\\
			\vdots	&	\vdots	&	\ddots	&	\vdots\\
			H_{n+k}	&	H_{n+k+1}	&	\cdots		&	H_{n+2k}\\
			\end{bmatrix}.
\end{equation*}
\end{definition}
\noindent Matrices of the form $H_{n, k}$ are known as Hankel matrices.

We need the following result of \cite{Sa}.

\begin{lemma}[Lemma 3 of \cite{Sa}]\label{lemma_salem}
A sequence satisfies some order-$k$ linear recurrence if and only if $\det(H_{n, k}) = 0$ for all $n$.
\end{lemma}

Let $\overline{H}_\ell = \{H_n\}_{n \geq \ell}$ denote the $\ell$th truncation of $H$, i.e., the sequence obtained by removing the first $\ell-1$ terms.

\begin{proposition}\label{prop_truncation}
Suppose $f$ is the minimal polynomial for $H$. Then $f$ is the minimal polynomial for all truncations $\overline{H}_\ell$.
\end{proposition}

\begin{proof}
Let $H$ be some linear recurrence sequence and let $f$ be the minimal polynomial for $H$. Let $f = x^t - c_1 x^{t-1} - \dots - c_t$. Note that $c_t \neq 0$. We shall show that $\det(H_{n, t-1}) \neq 0$ for all $n$. Let $D = \det(H_{1, t-1})$. In particular we will show that $|\det(H_{n, t-1})| = |c_t^n D|$.

We proceed by induction. Suppose we have
\begin{equation*}
H_{n, t-1} \ =\ 	\begin{bmatrix}
			H_n	&		H_{n+1}	&	\cdots	&		H_{n+t-1} \\
			H_{n+1}	&	H_{n+2}	&	\cdots	&		H_{n+t}\\
			\vdots	&	\vdots	&	\ddots	&	\vdots\\
			H_{n+t-1}	&	H_{n+t}	&	\cdots		&	H_{n+2t-2}\\
			\end{bmatrix}.
\end{equation*}
We index our columns starting from zero. Notice that the 1st through $(t-1)$th columns appear as columns in $H_{n+1, t-1}$. Furthermore, notice that if we multiply the zeroth column by $c_t$ and add to it $c_{t-1}$ times the first column, plus $c_{t-2}$ times the second column, plus etc., plus $c_1$ times the $t-1$ column, then the columns of the resulting matrix agree with the columns of $H_{n+1, t-1}$. The determinant has gone up by a factor of $c_t$.  In order to move the resulting matrix to the form of $H_{n+1, t-1}$, we need to permute some columns, which may change the sign of the determinant, but not the magnitude.

We know that $f \in \mathcal{I}(\overline{H}_n)$. However, if $f$ did not generate $\mathcal{I}(\overline{H_n})$, then there must be some polynomial of lower degree in $\mathcal{I}(\overline{H}_n)$. In particular, there must be some polynomial of degree $t-1$ in $\mathcal{I}(\overline{H}_n)$. If this were so, then $\det(H_{m, t-1})$ would be zero for all $m \geq n$ by Lemma \ref{lemma_salem}, which is a contradiction. Therefore, we must have that $f$ is the lowest degree polynomial in $\mathcal{I}(\overline{H}_n)$ for all $n$, and thus is the minimal polynomial for all $\overline{H}_n$.
\end{proof}

\begin{proof}[Proof of Proposition \ref{prop_min_poly}]
By Proposition $\ref{prop_truncation}$, it suffices to show that $\det{H_{-(t-1), t-1}} \neq 0$. It is immediate from writing out the matrix $H_{-(t-1),t-1}$ that by switching columns, we can make it lower triangular with diagonal entries all equal to $1$, and hence $\det({H_{-(t-1), t-1}})=1$.
\end{proof}

\section{Proof of Proposition \ref{prop_pos_type}} \label{appendix_pos_type}
\begin{proof}[Proof of Proposition \ref{prop_pos_type}]
Let $A$ denote the companion matrix corresponding to $f$. All entries in $A$ are non-negative. It is straightforward to check that the $t$th power of $A$ has all entries positive. Hence $A$ is a primitive matrix. Hence we may apply the Perron-Frobenius theorem to conclude that $f$ has a unique real dominating root of multiplicity one, call it $\beta$.

Notice that $\beta \geq c_1$ since $f(c_1) = -c_2 c_1^{t-2} - \dots - c_t \leq 0$. If $t \geq 2$, then $\beta > c_1 = 1$. If $c_1 \geq 2$, then $\beta \geq 2$. If $c_1 = 1$ and $t = 1$, then $f = x - 1$ which we do not consider to be of positive type.

Furthermore, notice that the coefficients of $f$ have exactly one sign change. Therefore by Descartes' rule of signs, $f$ must have at most one positive root, and hence exactly one.
\end{proof}
